\documentclass[a4paper,10pt]{article}
\usepackage{graphicx} % Required for inserting images

\usepackage{latexsym,amsmath,enumerate,amssymb,amsbsy,amsthm,lscape, hyperref} 
\usepackage {graphicx}
\usepackage{float,color,tikz}
\usepackage{subcaption}

\usepackage{algorithm}
\usepackage{algpseudocode}

\theoremstyle{plain}
\newtheorem{theorem}{Theorem}[section]
\newtheorem{proposition}[theorem]{Proposition}
\newtheorem{lemma}[theorem]{Lemma}
\newtheorem{corollary}[theorem]{Corollary}

\theoremstyle{definition}

\theoremstyle{remark}

\newtheorem{example}{Example}[section]

\title{Good Integers:  A Comprehensive Review with Applications}
\author{Somphong Jitman\footnote{S. Jitman is with the Department of Mathematics, Faculty of Science, Silpakorn University, Nakhon Pathom 73000, THAILAND}}
\date{October 17, 2025}

\begin{document}

\maketitle

\begin{abstract}
    For nonzero coprime integers $a$ and $b$, a positive integer $\ell$ is said to be \emph{good with respect to $a$ and $b$} if there exists a positive integer $k$ such that $\ell$ divides $a^{k} + b^{k}$. The concept of good integers has been the subject of continuous investigation since the 1990s due  to their elegant number-theoretic properties and their significant applications in various areas, particularly in coding theory.  This paper provides a comprehensive review of good integers, emphasizing both their theoretical foundations and their practical implications. We first revisit the fundamental number-theoretic properties of good integers and present their characterizations in a systematic manner. The exposition is enriched with well-structured algorithms and illustrative diagrams that facilitate their computation and classification.  Subsequently, we explore applications of good integers in the study of algebraic coding theory. In particular, their roles in the characterization,  construction, and enumeration  of self-dual cyclic codes and  complementary dual cyclic codes  are discussed in detail. Several examples are provided to demonstrate the applicability of the theory. This review not only consolidates existing results but also highlights the unifying role of good integers in bridging number theory and coding theory.

\end{abstract}

 \section{Introduction}	
\label{sec1-intro}
Good integers constitute a notable subfamily of positive integers that exhibit elegant and deep number-theoretic properties. Their structural richness has led to a variety of applications in both pure and applied mathematics. In particular, good integers play a significant role in the study of the algebraic properties of certain classes of polynomials, especially those arising in finite fields and modular arithmetic. Furthermore, they have been effectively employed in coding theory, where they contribute to the characterization and construction of specific families of codes, such as self-dual cyclic codes and complementary dual cyclic codes, under both the Euclidean and Hermitian inner products. The goal of this paper is to   provide a comprehensive review of good integers, emphasizing both their theoretical foundations and their practical implications. We mainly focus on presenting the main ideas and the connections among the results, while omitting the detailed proofs. For complete proofs, the reader is referred to the original references.

The concept of \emph{good integers} has been  formally introduced by P.~Moree in \cite{M1997} as part of  the study of  the divisors of  $a^k+b^k$. 
Given two fixed coprime nonzero integers $a$ and $b$, a positive integer $d$ is said  to be \emph{good (with respect to $a$ and $b$)} if there exists a positive integer $k$ such that $d$ divides $a^k+b^k$. 
 The integer  $d$ is called a \emph{bad integer} if no such $k$ exists.  To make the discussion more systematic,  let $G_{(a,b)}$ denote the set of all good integers defined with respect to the  integers  $a$ and $b$.  
The initial study for the case of  odd integers in  $G_{(a,b)}$ was established in \cite{M1997}, laying the foundation for subsequent research in this direction.  Even before the  formal treatment in \cite{M1997}, certain aspects of the set $G_{(q,1)}$, where $q$ is a prime power, had been investigated in \cite{KG1969}.
These early results were applied to the construction of BCH codes with prescribed design distances, thereby highlighting the relevance of good integers and applications in coding theory.  Building upon these connections, further properties of $G_{(q,1)}$ were subsequently investigated in \cite{S2003}, where the framework was employed to determine the average   hull dimension of cyclic codes, thus extending the scope of their applications.   The usefulness of good integers in coding theory continued to expand in the following years. 
In particular, the set $G_{(q,1)}$ was applied in the characterization of the existence of  some self-dual codes with respect to the Hermitian inner product over finite fields of square order, as reported in \cite{DMS2007} and \cite{DMS2014}. 
In another direction, the special case $G_{(2^\nu,1)}$ was employed to enumerate Euclidean self-dual cyclic codes of even length $n$ over finite fields $\mathbb{F}_{2^\nu}$ in \cite{JLX2011}.  Subsequently,  the good integers in $G_{(2^\nu,1)}$ were applied in the composition of group algebras   over  $\mathbb{F}_{2^\nu}$ which enabled the characterization and enumeration of self-dual abelian codes  defined under the Euclidean inner product  in \cite{JLLX2012}. 
 Further applications were presented in \cite{SJLP2015}, where $G_{(q,1)}$ played a central role in calculating the  hulls of  cyclic and negacyclic codes of arbitrary length  over $\mathbb{F}_q$. In this study,  the hull dimensions of such codes were established as well. 
More recently, attention has shifted to the study of good even integers, initiated in \cite{J2018a}, in which their characterization and number-theoretic properties were analyzed. 
Although some minor inaccuracies appeared in that work, they were subsequently corrected in \cite{J2018b} and \cite{JPR2020}, thereby refining the theoretical understanding of good even integers and completing the picture established in the earlier studies.

In \cite{J2018a}, the notion of good integers was refined by introducing two subclasses, namely the class of  {evenly-good integers} and the class of  {oddly-good integers}. 
For fixed nonzero coprime integers $a$ and $b$, a positive integer $d$ is called 
\emph{evenly-good (with respect to $a$ and $b$)} if $d$ divides $a^k+b^k$ for some even positive integer $k$, while it is called
\emph{oddly-good (with respect to $a$ and $b$)} if $d$ divides $a^k+b^k$ for some odd positive integer $k$. 
Hence, every good integer can be recognized as either  an  evenly-good integer or an  oddly-good integer. 
The corresponding sets are denoted by $EG_{(a,b)}$ and $OG_{(a,b)}$, respectively. 
It follows directly from the definition that
$
   G_{(a,b)} = OG_{(a,b)} \cup EG_{(a,b)}
$. Due to the commutativity $a^k+b^k=b^k+a^k$,   we have the relations: 
$G_{(a,b)}=G_{(b,a)}$, $OG_{(a,b)}=OG_{(b,a)}$, and $EG_{(a,b)}=EG_{(b,a)}$. 
A complete characterization of such two subclasses was established in \cite{J2018a}, providing a finer understanding of the structure of good integers.  Some preliminary investigations of these subclasses had already appeared earlier in \cite{JLS2014}, where  the sets $OG_{(2^\nu,1)}$ and $EG_{(2^\nu,1)}$ were examined and subsequently applied in the study of   self-dual abelian codes defined under the Hermitian inner product over $\mathbb{F}_{2^{2\nu}}$. 
Further progress was made in \cite{JS2016}, where the set $OG_{(q,1)}$ was utilized  in the study of  hulls   of cyclic codes over $\mathbb{F}_{q^2}$.  In this work,   the average Hermitian hull dimension of  such codes has been established using integers in $OG_{(q,1)}$.  
More recently, these concepts have also played a role in the analysis of special factors of the polynomial $x^n - 1$ and  $x^n + 1$ over finite fields in  \cite{BJ2021}, \cite{BJ2021-s}, and  \cite{BJU2019}.

The paper is organized as follows.   Section~\ref{sec2} recalls the definitions, characterizations, and key properties of good integers. This section also presents a diagrammatic framework and a useful algorithm for identifying and classifying good integers.  Section~\ref{sec3} is devoted to a detailed investigation of two notable families of good integers, namely the the class of evenly-good  integers and the class of oddly-good integers. 
In this section, their distinctive algebraic features are analyzed and more refined characterizations are established. 
To complement the theoretical results, diagrammatic illustrations and a specialized algorithm adapted to these subclasses are also presented, with particular emphasis on their potential utility in algebraic studies and applications to coding theory.   Section~\ref{sec4} turns to a survey of selected applications of good integers, with special attention to oddly-good integers. 
These applications include their role in examining the algebraic structure and factorization of certain polynomial families over finite fields as well as their significance in the construction, characterization,  and enumeration of several classes of cyclic codes over finite fields. 
In particular, connections to self-dual and complementary dual cyclic codes are highlighted, both in the Euclidean and Hermitian settings.   Finally, Section~\ref{sec5}  presents a summary of the main findings and provides perspectives for future research in this area.

\section{Good Integers} \label{sec2}
This section presents a concise summary of the fundamental properties and characterization of good odd integers, as originally established in \cite{M1997}, together with the properties of good even integers that were subsequently developed in \cite{J2018a},   \cite{J2018b}, and \cite{JPR2020}.  
  To enhance visualization and support further applications, we include a diagram that illustrates the characterization of these good integers. Additionally, we provide an algorithm designed to facilitate the practical computation and identification of good integers.

For an integer $i$ and a nonzero integer    $j$, we write $j |i$ if $j$ divides $i$, and 
we use the notation $2^i || j$ to indicate that  $i$ is the largest nonnegative integer with the property that $2^i | j$.  
For a positive integer $N$, let $\mathrm{ord}(a)$ denote the order of an element $a$ in the additive group $\mathbb{Z}_N$. 
For integers  $a$ and $b$    coprime to $N$, the multiplicative  of $b$ modulo $N$ exists, we denote it by   $b^{-1}$.  Let $\mathrm{ord}_N(a)$ denote the multiplicative order of $a$ modulo $N$ and let 
$
   \mathrm{ord}_N\!\left(\tfrac{a}{b}\right) :=  \mathrm{ord}_N\!\left( ab^{-1}\right) $.

\subsection{Good Odd Integers}
\label{subsec2.1}

This subsection presents the complete characterization of good odd integers, as established in~\cite{M1997}. These foundational results are instrumental for the development of the theory of good even integers in the next subsection and for characterizing certain subclasses of good integers in Section \ref{sec3}.

From the definition of $G_{(a,b)}$, it is immediate that $1$ is always a good integer.  The properties of other good odd integers have been extensively established in \cite{M1997} and \cite{J2018a}. The key findings and significant results from these studies are summarized below.

\begin{lemma}[{\cite[Lemma 2.1]{J2018a}}]
	\label{gcd-1-ab} Let $a$ and $ b$  be nonzero coprime  integers and let $d$ be a positive    integer. If    $d\in G_{(a,b)}$, then     $\gcd(a,d)=1=\gcd(b,d)$.  
\end{lemma}

 %The following lemma from \cite{M1997} is useful.
% \begin{lemma}
% 	[{\cite[Lemma 1]{M1997}}]\label{congr} Let $a_1,a_2,\dots,a_t$ be positive integers. Then the system of congruences 
% 	\begin{align*}
% 		x\equiv a_1\,({\rm mod}\, 2a_1), \quad x\equiv a_2\,({\rm mod}\, 2a_2), \quad \dots, \quad x\equiv a_t\,({\rm mod}\, 2a_t) 
% 	\end{align*}
% 	has a solution $x$ if and only if there exists $s\geq 0$ such that $2^s|| a_i$ for all $i=1,2,\dots, t$. 
% \end{lemma}

Based on Lemma~\ref{gcd-1-ab},  it is important to note that each integer in $G_{(a,b)}$ must be coprime to both integers $a$ and $b$. From now on, suppose that $a$ and $b$ are nonzero coprime integers, and that all underlying integers under consideration are coprime to both $a$ and $b$.

In the sequel, we establish a number of fundamental results concerning good odd prime powers.

\begin{proposition}
	[{\cite[Proposition 2]{M1997}}]\label{2order} 
    Let $p$ be an odd prime and let $r$ be a positive integer. If $p^r \in G_{(a,b)}$, then ${\rm ord}_{p^r}(\frac{a}{b})=2s$, where $s$ is the smallest positive integer such that $({a}{b}^{-1})^s\equiv -1 \,({\rm mod}\, p^{r})$. 
\end{proposition}
\begin{proposition}
	[{\cite[Proposition 4]{M1997}}]\label{ord} Let $p$ be an odd prime and let $r$ be a positive integer. Then ${\rm ord}_{p^r}(\frac{a}{b})={\rm ord}_{p}(\frac{a}{b})p^i$ for some $i\geq 0$. 
\end{proposition}
From the propositions above, it follows that both 
$  \mathrm{ord}_{p^r}\!\left(\tfrac{a}{b}\right)  $ and $ \mathrm{ord}_{p}\!\left(\tfrac{a}{b}\right)$
are even for every odd prime power $p^r \in G_{(a,b)}$. 
Furthermore, by Proposition~\ref{2order}, one obtains  
\[
   \left(a b^{-1}\right)^{\frac{\mathrm{ord}_{p^r}(\tfrac{a}{b})}{2}} \equiv -1 \pmod{p^r},
 \quad \text{ and hence, }\quad
   p^r \;\bigm|\; \left( a^{\tfrac{\mathrm{ord}_{p^r}(\tfrac{a}{b})}{2}} + b^{\tfrac{\mathrm{ord}_{p^r}(\tfrac{a}{b})}{2}} \right).
\]

By examining the prime divisors of an odd integer $d>1$, we arrive at the following theorem, which provides a complete  number-theoretic characterization of good odd integers.

\begin{theorem}
	[{\cite[Theorem 1]{M1997}}]\label{goodP} Let $d$ be an odd integer such that $d>1$. Then $d\in G_{(a,b)}$   if and only if there exists a positive integer $s$ such that $2^s|| {\rm ord}_{p}(\frac{a}{b})$ for every prime $p$ dividing~$d$. 
\end{theorem}
Theorem \ref{goodP} and Proposition \ref{2order} ensure that the order   ${\rm ord}_{d}(\frac{a}{b})$ is even  and  $({a}{b}^{-1})^\frac{{\rm ord}_{d}(\frac{a}{b})}{2}\equiv -1 \,({\rm mod}\, d)$  for all  odd integers $d\in G_{(a,b)}$.  Using these properties,   it can be obtained that 
\[d\big\vert \left(a^\frac{{\rm ord}_{d}(\frac{a}{b})}{2} +b^\frac{{\rm ord}_{d}(\frac{a}{b})}{2} \right).\]

\begin{example} \label{ex1} An illustrative  identification  of  good odd integers with respect to $a=11$ and $b=12$ is given as follows. 
 \begin{enumerate}
     \item Let $d= 1625= 5^3\cdot 13$.  Then $b^{-1} \equiv 948 \pmod d$  which implies that ${\rm ord}_{5}(\frac{a}{b}) =4$ and   ${\rm ord}_{13}(\frac{a}{b})=12$.  It follows that  $2^2||{\rm ord}_{5}(\frac{a}{b}) $ and   $2^2||{\rm ord}_{13}(\frac{a}{b})$ which mean $d\in G_{(a,b)}$. Since ${\rm ord}_{d}(\frac{a}{b}) =300$, we have $d|(a^{150}+b^{150})$.

     \item Let $d= 6125= 5^3\cdot 7^2$.  Then $b^{-1} \equiv 3573 \pmod d$  which implies that ${\rm ord}_{5}(\frac{a}{b}) =4$ and   ${\rm ord}_{7}(\frac{a}{b})=6$.  It follows that  $2^2||{\rm ord}_{5}(\frac{a}{b}) $ but    $2^1||{\rm ord}_{13}(\frac{a}{b})$.  Hence,  $d\notin G_{(a,b)}$.

     \item  Let $d= 3875= 5^3\cdot 31$.  Then $b^{-1} \equiv 323 \pmod d$  which implies that ${\rm ord}_{5}(\frac{a}{b}) =4$ and   ${\rm ord}_{31}(\frac{a}{b})=15$.  It follows that  $2^2||{\rm ord}_{5}(\frac{a}{b}) $ but    $2^0||{\rm ord}_{31}(\frac{a}{b})$.  Therefore, we have  $d\notin G_{(a,b)}$.  
 \end{enumerate} 
 \end{example}

\subsection{Good Even Integers}
\label{subsec2.2}
 In \cite{M1997}, the characterization of the  existence of an  good even integer has   been briefly explained. Precisely, a good even integer exists if and only if $ab$ is odd, which is now formally established by Lemma~\ref{gcd-1-ab}. The further characterization and properties of good even integers was  established  in \cite{J2018a}, although some minor errors were identified in the statements of \cite[Proposition 2.1 and Proposition 2.3]{J2018a}. The necessary corrections were later provided in \cite{J2018b} together with there proof in \cite{JPR2020}.

\begin{proposition}[{\cite[Proposition 2.3]{JPR2020}}] \label{evengood}  
	Let $a$ and $b$   be  coprime odd  integers  and  let $\beta$ be a positive integer. Then the following statements are equivalent.
	\begin{enumerate}[$i)$]
		\item
		$2^\beta\in G_{(a,b)}$.
		\item      $2^\beta|(a+b)$.
		\item  $ab^{-1} \equiv -1 ~{\rm mod}~ 2^\beta$. 
	\end{enumerate}
\end{proposition}

\begin{proposition}[{\cite[Proposition 2.2]{J2018a}}]   \label{prop2d}
	Let $a$, $b$  and $d>1$ be  pairwise coprime odd     integers.  Then    $d\in G_{(a,b)}$
	if and only if        $2d\in G_{(a,b)}$. 
	In this case, 	   ${\rm ord}_{2d}(\frac{a}{b})= {\rm ord}_{d}(\frac{a}{b})$ is even. 
\end{proposition}
From Proposition \ref{prop2d},  it can be deduced further that
$({a}{b}^{-1})^\frac{{\rm ord}_{d}(\frac{a}{b})}{2}\equiv -1 \,({\rm mod}\, d)$  and  $({a}{b}^{-1})^\frac{{\rm ord}_{d}(\frac{a}{b})}{2}\equiv -1 \,({\rm mod}\, 2d)$   for all   odd integers $d\in G_{(a,b)}$. We conclude the following divisibility: 
\[d\big\vert \left(a^\frac{{\rm ord}_{d}(\frac{a}{b})}{2} +b^\frac{{\rm ord}_{d}(\frac{a}{b})}{2} \right) \text{ and } 2d\big\vert \left(a^\frac{{\rm ord}_{d}(\frac{a}{b})}{2} +b^\frac{{\rm ord}_{d}(\frac{a}{b})}{2} \right).\]

\begin{proposition}[{\cite[Proposition 2.7]{JPR2020}}]  \label{prop2m} Let $a,b$ and $d>1$   be pairwise coprime odd positive integers  and let  $\beta\geq 2$ be  an  integer. Then    $2^\beta d\in G_{(a,b)}$ if and only if   $2^\beta|(a+b)$
	and  $d\in G_{(a,b)}$  is such that  $2|| {\rm ord}_{d}(\frac{a}{b})$.  
	In this case,    ${\rm ord}_{2^\beta}(\frac{a}{b})=2$  and   $2|| {\rm ord}_{2^\beta d}(\frac{a}{b})$.
\end{proposition}
  
 From Proposition \ref{prop2m},  it can be deduced   that
$({a}{b}^{-1})^\frac{{\rm ord}_{2^\beta d}(\frac{a}{b})}{2}\equiv -1 \,({\rm mod}\, 2^\beta d)$   for all     integers $2^\beta d\in G_{(a,b)}$. This  implies that 
\[2^\beta d \big\vert \left(a^\frac{{\rm ord}_{2^\beta d}(\frac{a}{b})}{2} +b^\frac{{\rm ord}_{2^\beta d}(\frac{a}{b})}{2} \right)  .\]

The results can be obtained immediately from Propositions~\ref{evengood}--\ref{prop2m}. 
Alternatively, it may be viewed as a corrected and refined version of {\cite[Theorem~2.1]{J2018a}}.

\begin{theorem} \label{thm2} Let $a$ and $b$    be     coprime  nonzero   integers and let $\ell=2^\beta d$ be a positive integer such that $d$ is odd and $\beta \ge 0$.  Then one of the following statements holds.
     \begin{enumerate}[$1)$]
         \item  If  $ab$ is odd, then $\ell=2^\beta d\in G_{(a,b)}$  if and only if  one of the following statements holds.
         \begin{enumerate}[$(a)$]
             \item $\beta \in\{0,1\}$ and  $d=1$.
             \item $\beta \in \{0,1\}$, $d\geq 3$   and  
             there exists $s\geq 1$ such that $2^s|| {\rm ord}_{p}(\frac{a}{b})$ for every prime $p$ dividing $d$. 
             \item {$\beta \ge 2$, $d=1$  and    $ 2^\beta |(a+b)$.}
             \item $\beta \ge  2$, $d\geq 3$,     $ 2^\beta |(a+b)$  and  $ d\in G_{(a,b)}$ is such that $2|| {\rm ord}_{d}(\frac{a}{b})$.   
         \end{enumerate}
         \item  If $ab$ is even,  then $\ell=2^\beta d\in G_{(a,b)}$  if and only if    one of the following statements holds.
         \begin{enumerate}[$(a)$]
             \item $\beta =0$   and $d=1$.
             \item  $\beta =0$, $d\geq 3$,   and 
             there exists $s\geq 1$ such that $2^s|| {\rm ord}_{p}(\frac{a}{b})$ for every prime $p$ dividing $d$. 
         \end{enumerate}
     \end{enumerate}
 \end{theorem}

\begin{example}    \label{ex2}
 
An illustrative identification of good even integers with respect to $a=5$ and $b=7$ is given as follows.  
Let $d = 1573 = 11^2 \cdot 13$. Then $b^{-1} \equiv 899 \pmod d$ which implies that 
${\rm ord}_{11}\!\left(\frac{a}{b}\right) = 10$ and 
${\rm ord}_{13}\!\left(\frac{a}{b}\right) = 6$. 
It follows that $2^1 \Vert {\rm ord}_{11}\!\left(\frac{a}{b}\right)$ and 
$2^1 \Vert {\rm ord}_{13}\!\left(\frac{a}{b}\right)$, which means 
$d \in G_{(a,b)}$ by Theorem~\ref{goodP}.

\begin{enumerate}
    \item Since ${\rm ord}_d\!\left(\frac{a}{b}\right) = 330$, we have 
    $d |\left(a^{165} + b^{165}\right)$. By Proposition~\ref{prop2d}, 
    $2d = 3146 \in G_{(a,b)}$ and $2d |\left(a^{165} + b^{165}\right)$.

    \item Since $2^2 |(a+b)$ and $d \in G_{(a,b)}$ with 
    $2^1 \Vert {\rm ord}_d\!\left(\frac{a}{b}\right)$, it follows from 
    Proposition~\ref{prop2m} that $2^2 d =6292 \in G_{(a,b)}$. 
    Moreover, since ${\rm ord}_{2^2 d}\!\left(\frac{a}{b}\right) = 330$, we obtain 
    $2^2 d |\left(a^{165} + b^{165}\right)$.

    \item Since $2^\beta \nmid (a+b)$ for all $\beta \geq 3$, we conclude from 
    Proposition~\ref{prop2m} that $2^\beta d \notin G_{(a,b)}$ for all $\beta \geq 3$.
\end{enumerate}

\end{example}

 \subsection{Useful Algorithm and Diagrams}

From the previous subsections,  the following diagrams provide an illustrative view for the characterizations of good integers.  The result in blue  is recalled  from \cite{M1997}  (see Theorem \ref{goodP}) and the others  summarized from those in   \cite{J2018b} and \cite{J2018a}.  Based on Lemma \ref{odd-good2}, the diagrams are presented separately for the case where $ab$ is odd  in  Figure \ref{Fig1} and where $ab$ is even in Figure \ref{Fig2}.  To save space,  we write $G$ instead of $G_{(a,b)}$.
 
	\begin{figure}[!hbt]
\begin{tikzpicture} [ 
% auto,
% decision/.style = { diamond, draw=blue, thick, fill=blue!20,
% 	text width=5em, text badly centered,
% 	inner sep=1pt, rounded corners },
block/.style    = { rectangle, draw=black, 
    fill=black!5, text width=11em, text centered,
    rounded corners, minimum height=2em },
line/.style     = { draw, thick, ->, shorten >=2pt }
]
% Define nodes in a matrix
\matrix [column sep=2mm, row sep=12 mm] {
    & \node [block] (ell) {$\ell$ };            && \\ %  
    \node [block] (ell1) {$1,2\in G$ };&&  &\node [block] (ell6) {$2^\beta d\in G$ iff $2^\beta,d\in G$ and $2||$ord$_d(ab^{-1}$) }; \\
   \node [block] (ell3) {$2^\beta \in G$ iff $2^\beta|(a+b) $}; &  \node [block] (ell4) {\color{blue} $d\in G$ iff $\exists s\geq 1,   \forall p|d$, $2^s||$ord$_p(ab^{-1}$) }; &
    \node [block] (ell5) {$2d\in G$ iff  $d\in G$ };  &\\
};
\begin{scope} [every path/.style=line]
\path (ell)        --   node{\quad \quad $1,2$}   (ell1);  
%\path (ell)        --   node{\quad\quad $2$}   (ell2);  
\path (ell)        --   node{\quad\quad \quad   $2^\beta, ~\beta\geq 2$}   (ell3);  
\path (ell)        --   node{\quad \quad \quad \quad $d$ is odd.}   (ell4);  
\path (ell)        --   node{\quad \quad \quad \quad\quad $2d,~d$ is odd.}   (ell5);  
\path (ell)        --   node{\quad\quad \quad \quad\quad \quad \quad $2^\beta d,~d$ is odd and $\beta\geq 2$.}   (ell6);  
\end{scope}
%
% legend for subprocedures
% legend for input and output variables
\end{tikzpicture}

\caption{Characterization of good integers with respect to coprime  integers   $a$ and $b$, where  $ab$ is odd.} \label{Fig1}
\end{figure}
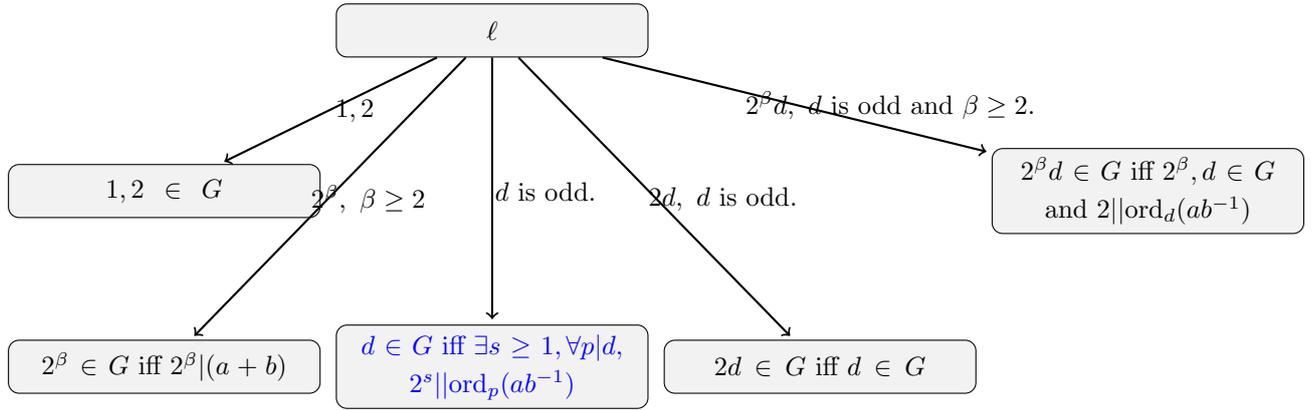

 \begin{figure}[!hbt]
\begin{tikzpicture} [
% auto,
% decision/.style = { diamond, draw=blue, thick, fill=blue!20,
% 	text width=5em, text badly centered,
% 	inner sep=1pt, rounded corners },
block/.style    = { rectangle, draw=black, 
    fill=black!5, text width=11em, text centered,
    rounded corners, minimum height=2em },
line/.style     = { draw, thick, ->, shorten >=2pt },
]
% Define nodes in a matrix
\matrix [column sep=2mm, row sep=12  mm] {
    
    & \node [block] (ell) {$\ell$ };            && \\ %  
    
    \node [block] (ell1) {$1$ is $G$ };&&  &\node [block] (ell6)  {$2^\beta d $ is bad };\\
    
    \node [block] (ell2)  {$2$ is bad};&  \node [block] (ell3)  {$2^\beta  $ is bad}; &\node [block] (ell4) {\color{blue} $d\in G$ iff $\exists s\geq 1, \forall p|d$, $2^s||$ord$_p(ab^{-1}$) };&\node [block] (ell5)  {$2d\ $  is bad};  \\
};

\begin{scope} [every path/.style=line]
\path (ell)        --   node{\quad \quad $1$}   (ell1);  

\path (ell)        --   node{\quad\quad $2$}   (ell2);  

\path (ell)        --   node{\quad\quad \quad   $2^\beta, ~\beta\geq 2$}   (ell3);  

\path (ell)        --   node{\quad \quad \quad \quad $d$ is odd.}   (ell4);  

\path (ell)        --   node{\quad \quad \quad \quad\quad $2d,~d$ is odd.}   (ell5);  

\path (ell)        --   node{\quad\quad \quad \quad\quad \quad \quad $2^\beta d,~d$ is odd and $\beta\geq 2$.}   (ell6);

\end{scope}
%
% legend for subprocedures
% legend for input and output variables
\end{tikzpicture}
\caption{Characterization of good integers with respect to coprime  integers   $a$ and $b$, where  $ab$ is   even.} \label{Fig2}
\end{figure}
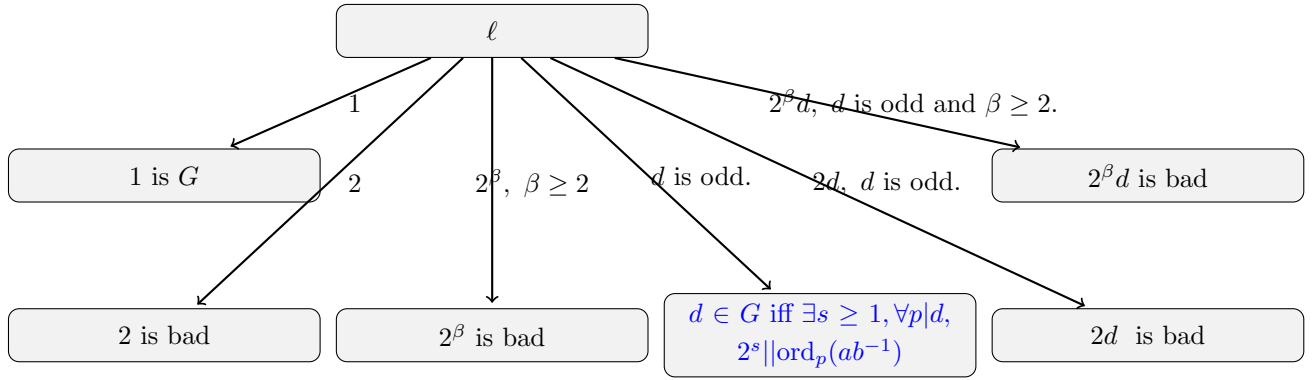 

Based on the characterizations presented in Subsections~\ref{subsec2.1}--\ref{subsec2.2}, we summarize Algorithm~\ref{Alg1}, which provides an effective method for identifying good integers. This algorithm translates the theoretical criteria into a practical procedure that facilitates computation and supports further applications in number theory and coding theory.

\begin{algorithm}
	\begin{algorithmic}
		\Require Integers \(a\), \(b\), and \(\ell\geq 1\) such that $\gcd(a,b)=1$
		\If{\(\gcd(a,\ell) \ne 1\) or \(\gcd(b,\ell) \ne 1\)}
		\Return \(\ell \notin G_{(a,b)}\)  
		\Else
		\State Write \(\ell = 2^\beta d\) where \(d\) is odd and \(\beta \geq 0\)
		\If{\(d = 1\)}
		\If{\(2^\beta \nmid (a + b)\)}
		\Return \(\ell \notin G_{(a,b)} \) 
		 
		\Else ~ \Return \(\ell \in G_{(a,b)} \)   
        %[\(\because\)  \(\ell |(a^1 + b^1)\).]
		\EndIf
		\Else
		\State Factor \(d = p_1^{r_1}p_2^{r_2}\dots p_t^{r_t}\)
		\State Compute \(\text{ord}_{p_i}(ab^{-1})\) for \(i = 1, \dots, t\)
		\State Set \(S = \{s_i : 2^{s_i} \Vert \text{ord}_{p_i}(ab^{-1})\}\)
		\If{\(\beta \geq 1\) and \(ab\) is even}
		\Return \(\ell \notin G_{(a,b)}\)  
		\ElsIf{\(\beta = 0\)}
		\If{\(S = \{0\}\) or \(|S| \geq 2\)}
		\Return \(\ell \notin G_{(a,b)}\)  
		\Else ~
		 \Return \(\ell \in EG_{(a,b)}\)  
	%	\(\ell |\left(a^{\text{ord}_d(ab^{-1})/2} + b^{\text{ord}_d(ab^{-1})/2}\right)\)
		\EndIf
		\ElsIf{\(\beta = 1\) and \(ab\) is odd}
		\If{\(S = \{0\}\) or \(|S| \geq 2\)}
		\Return \(\ell \notin G_{(a,b)}\)  
		 
		\Else ~
		\Return \(\ell \in G_{(a,b)}\)  
	%	\(\ell |\left(a^{\text{ord}_d(ab^{-1})/2} + b^{\text{ord}_d(ab^{-1})/2}\right)\)
		\EndIf 
		\Else
		\If{\(S = \{1\}\) and \(2^\beta |(a + b)\)}
		\Return \(\ell \in G_{(a,b)}\)  
		%\(\ell |\left(a^{\text{ord}_\ell(ab^{-1})/2} + b^{\text{ord}_\ell(ab^{-1})/2}\right)\)
		\Else ~
		\Return \(\ell \notin G_{(a,b)}\)  
		\EndIf
		\EndIf
		\EndIf
		\EndIf
	\end{algorithmic}
    	\caption{Identification of the good  integers     with respect to coprime  integers  \(a\) and \(b\)} \label{Alg1}
\end{algorithm}

Based on the procedure described in Algorithm~\ref{Alg1}, illustrative computational results for good integers less than $50$ in  $G_{(a,b)}$ with respect to $a,b\in \{1,2,3,4,5,6\}$ are presented in Table~\ref{T1}. This table demonstrates the algorithm  in identifying good integers and provides concrete examples that support the theoretical framework discussed earlier.

\begin{table}[!hbt]
    \centering
    \begin{tabular}{ccl}
    \hline
       $a$  & $b$  & Good integers in $G_{(a,b)}$\\
         \hline
$1$& $2$ &$ 1,
3 ,
5 ,
9 ,
11 ,
13 ,
17 ,
19 ,
25 ,
27 ,
29 ,
33 ,
37 ,
41 ,
43 , \dots$\\
$1$ & $3$ & $1,2,
4 ,
5 ,
7 ,
10 ,
14 ,
17 ,
19 ,
25 ,
28 ,
29 ,
31 ,
34 ,
37 ,
38 ,
41 ,
43 ,
49 ,
 \dots $\\
$1$ &$4$& $1,
5 ,
13 ,
17 ,
25 ,
29 ,
37 ,
41 , \dots$ \\
$1$& $5$& $1,2,
3 ,
6 ,
7 ,
9 ,
13 ,
14 ,
17 ,
18 ,
21 ,
23 ,
26 ,
27 ,
29 ,
34 ,
37 ,
41 ,
42 ,
43 ,
46 ,
47 ,
49 ,\dots$\\

$1$& $6$ & $1,
7 ,
11 ,
13 ,
17 ,
29 ,
31 ,
37 ,
41 ,
49 , \dots $\\
$2$ &$3$&$ 1,
5 ,
7 ,
11 ,
13 ,
17 ,
25 ,
31 ,
35 ,
37 ,
41 ,
49 , \dots$\\
$2$&$ 5$ &$1,
7 ,
11 ,
17 ,
19 ,
23 ,
29 ,
37 ,
41 ,
47 ,
49 , \dots$\\
$3$ & $4$& $1,
5 ,
7 ,
13 ,
17 ,
19 ,
25 ,
29 ,
31 ,
41 ,
43 ,
49 , \dots$\\
$3$& $5$ &$1,2,
4 ,
8 ,
13 ,
17 ,
19 ,
23 ,
26 ,
29 ,
31 ,
34 ,
37 ,
38 ,
41 ,
46 ,
47 , \dots$\\
$4$ &$5$ &$1,
3 ,
7 ,
9 ,
13 ,
17 ,
21 ,
23 ,
27 ,
37 ,
41 ,
43 ,
47 ,
49 , \dots$\\
$5$& $6$ & $1,
11 ,
17 ,
23 ,
31 ,
41 ,
43 ,
47 ,\dots$ \\
         \hline 
    \end{tabular}
    \caption{Good integers}\label{T1}
    \label{tab:placeholder}
\end{table}

\section{Oddly-Good and Evenly-Good Integers} \label{sec3}

Two notable subclasses of good integers, namely  the class of {evenly-good} integers and the class of {oddly-good} integers, which were first formalized in \cite{J2018a}.  The reader is referred to Section \ref{sec1-intro} for precised definitions. 
The distinction between these subclasses arises naturally from whether a divisor $d$ of $a^k+b^k$ is witnessed by an odd or an even exponent $k$, leading to structural differences that are not visible when considering good integers as a whole. 
As will be seen, these refinements not only yield a more detailed characterization of the set $G_{(a,b)}$, but also reveal additional number-theoretic properties with significant implications in applications. 
In particular, the separation into oddly-good and evenly-good integers has proven useful in the analysis of polynomial factorization over finite fields and in the characterization of certain classes of cyclic   codes over finite fields.   To enhance comprehension and facilitate further applications, a diagram is provided to illustrate their characterizations. In addition, an algorithm is presented to support the   computation and identification of these  good integers.

It is immediate to verify that $1|(a+b)$ and $1|(a^2+b^2)$, which shows that the integer $1$ is trivially both  evenly-good and oddly-good.  
For the case $\ell=2$, the situation is slightly more delicate. 
The integer $2$ is good \emph{if and only if} $ab$ is odd. 
In this case, $a+b$ must be even, which forces $2|(a+b)$ and $2|(a^2+b^2)$. 
Thus, whenever $2$ is good, it also belongs simultaneously to $EG_{(a,b)}$ and $OG_{(a,b)}$.   In contrast, for any integer $\ell>2$, the situation changes fundamentally. 
It is no longer possible for the same integer $\ell$ to satisfy both characterizations at once; that is, no such $\ell$ can be both an evenly-good  integer and an oddly-good integer  simultaneously.

 \begin{proposition}[{\cite[Proposition 3.1]{J2018a}}]
 	\label{disj} Let $a$, $b$ and $\ell>2$ be   pairwise coprime  nonzero integers. If   $\ell\in G_{(a,b)}$, then either $\ell \in OG_{(a,b)} $ or $\ell \in EG_{(a,b)}$, but not both. 
 \end{proposition}
 Based on Proposition \ref{disj}, we present the characterization and fundamental properties of evenly-good and oddly-good integers in a systematic manner, treating them separately in the next two sections.

 \subsection{Oddly-Good Integers}

In this subsection, our attention is devoted to the subclass of \emph{oddly-good} integers, originally investigated in \cite{J2018a}.  
We provide a complete characterization of these integers and discuss their fundamental algebraic and number-theoretic properties in detail.  
To supplement the theoretical development, a number of illustrative examples are also presented, serving to demonstrate the characterization in explicit and concrete cases.

 \begin{proposition}[{\cite[Proposition 3.2]{J2018a}}]
 	\label{odd-good} Let $a$ and $b$ be coprime nonzero integers and let $d>1$ be an  odd  integer. Then   $d\in OG_{(a,b)}$   if and only if $2|| {\rm ord}_{p}(\frac{a}{b})$ for every prime $p$ dividing $d$.  
 \end{proposition}

From Proposition~\ref{odd-good},   the following corollary can be obtained directly, which gives an alternative characterization of oddly-good odd integers.  Beyond offering a clear criterion for their identification, this result further shows that the collection of oddly-good odd integers is closed under usual multiplication, thereby highlighting an intrinsic algebraic structure within this subclass.

 \begin{corollary}[{\cite[Corollary 3.1]{J2018a}}]
 	\label{cor-oddly} Let $a$ and $b$ be coprime nonzero integers and let $d>1$ be an  odd  integer. Then the following statements are equivalent. 
 	\begin{enumerate} 
 		\item $d\in OG_{(a,b)}$. 
 		\item $j\in OG_{(a,b)}$ for all divisors $j$ of $d$. 
 		\item $p\in OG_{(a,b)}$  for all prime divisors $p$ of $d$. 
 	\end{enumerate}
 \end{corollary}

  Together with  Propositions \ref{prop2d} and  \ref{prop2m},    the characterization   oddly-good  even integers is given in the following corollary. 
 \begin{corollary}[{\cite[Corollary 3.2]{J2018a}}]
 	\label{odd-good2} Let $a$ and $b$ be coprime nonzero integers and let $d>1$ be an  odd  integer.   
 	\begin{enumerate}
 		\item  The following statements are equivalent.
 		\begin{enumerate}
 			\item $d\in OG_{(a,b)}$.
 			\item       $2d\in OG_{(a,b)}$.  
 		\end{enumerate} 
 		\item  For each $\beta \ge 2$,  $2^\beta d\in OG_{(a,b)}$ if and only if $2^\beta d\in G_{(a,b)}$.
 	\end{enumerate}    
 \end{corollary}

Based on the discussion above, we summarize the characterization and fundamental properties of arbitrary oddly-good integers.  This result may also be regarded as a corrected and refined version of {\cite[Theorem~3.1]{J2018a}}.

\begin{theorem} \label{thm3}  Let $a$ and $b$     be   coprime  nonzero integers  and  let $\ell =2^\beta d$ be  an  integer such that $d$ is odd and $\beta \ge 0$. Then  one of the following statements holds.
     \begin{enumerate}[$1)$]
         \item  If  $ab$ is odd, then $\ell=2^\beta d\in OG_{(a,b)}$  if and only if  one of the following statements holds.
         \begin{enumerate}[$(a)$]
             \item $\beta \in \{0,1\}$ and  $d=1$.
             \item 
             $\beta \in \{0,1\}$, $d\geq 3$,   and  $2|| {\rm ord}_{p}(\frac{a}{b})$ for every prime $p$ dividing $d$. 
             \item $\beta \ge 2$, $d=1$  and  $2^\beta |(a+b)$. 
             \item $\beta \ge 2$, $d\geq 3$,      $ 2^\beta |(a+b)$   and  $ d\in G_{(a,b)}$ is such that $2|| {\rm ord}_{d}(\frac{a}{b})$.   
         \end{enumerate}
         \item  If $ab$ is even,  then $\ell=2^\beta d\in OG_{(a,b)}$  if and only if    one of the following statements holds.
         \begin{enumerate}[$(a)$]
             \item  $\beta =0$   and $d=1$.
             \item  
             $\beta =0$, $d\geq 3$,   and  $2|| {\rm ord}_{p}(\frac{a}{b})$ for every prime $p$ dividing $d$. %} 
         \end{enumerate}
     \end{enumerate}
 \end{theorem}

For each $2^\beta d\in OG_{(a,b)}$,   we have $2|| {\rm ord}_{2^\beta  d}(\frac{a}{b})$ by  Proposition \ref{prop2m} and Theorem  \ref{thm3}.  It follows that  $\frac{{\rm ord}_{2^\beta d}(\frac{a}{b})}{2}$ is odd  and 
$({a}{b}^{-1})^\frac{{\rm ord}_{2^\beta d}(\frac{a}{b})}{2}\equiv -1 \,({\rm mod}\, 2^\beta d)$   for all     integers $2^\beta d\in G_{(a,b)}$. As desired,  we have 
\[2^\beta d \big\vert \left(a^\frac{{\rm ord}_{2^\beta d}(\frac{a}{b})}{2} +b^\frac{{\rm ord}_{2^\beta d}(\frac{a}{b})}{2} \right)  .\]

 \begin{example}    Let $a=5$  and $b=7$.   From Example \ref{ex2}, we have   $d = 1573 = 11^2 \cdot 13 \in G_{(a,b)}$ with 
  $2^1 \Vert {\rm ord}_{11}\!\left(\frac{a}{b}\right)$ and 
$2^1 \Vert {\rm ord}_{13}\!\left(\frac{a}{b}\right)$. By Proposition \ref{odd-good}, we have  $d = 1573 = 11^2 \cdot 13 \in OG_{(a,b)}$   with   ${\rm ord}_d\!\left(\frac{a}{b}\right) = 330$ and 
    $d |\left(a^{165} + b^{165}\right)$.
Since   $2|(a+b)$ and $2^2 |(a+b) $,  it follows that  $2d=3146  , 2^2 d =6292 \in OG_{(a,b)}$. 
As ${\rm ord}_{2d}\!\left(\frac{a}{b}\right) = 330= {\rm ord}_{2^2d}\!\left(\frac{a}{b}\right)$, we have 
    $2d |\left(a^{165} + b^{165}\right)$ and $2^2d |\left(a^{165} + b^{165}\right)$. Since $2^\beta \nmid (a+b)$ for all $\beta \geq 3$, we conclude    $2^\beta d \notin OG_{(a,b)}$ for all $\beta \geq 3$ by Theorem \ref{thm3}.
\end{example}
 \subsection{Evenly-Good Integers} 

Here, we turn our attention to the subclass of {evenly-good} integers, as introduced in \cite{J2018a}.  
We present a complete characterization of these integers and examine their fundamental structural properties.  
To complement the theoretical results, a number of illustrative examples are also provided.

From Proposition~\ref{disj}, it follows that the sets of evenly-good and oddly-good integers can overlap only in the cases $\ell=1$ or $\ell=2$.  
By combining Theorem~\ref{goodP} with Proposition~\ref{odd-good}, we obtain the following characterization and number-theoretic properties of evenly-good odd integers.

 \begin{proposition}[{\cite[Proposition 3.3]{J2018a}}]
 	\label{odd-even} Let $a$ and $b$ be coprime nonzero  integers and let $d>1$ be an  odd  integer. Then  
 	$d\in EG_{(a,b)}$  if and only if there exists $s\geq 2$ such that $2^s|| {\rm ord}_{p}(\frac{a}{b})$ for every prime $p$ dividing $d$.  
 \end{proposition}

By combining   Theorem~\ref{thm2},  Theorem~\ref{thm3}, and Proposition \ref{odd-even}  we obtain the following comprehensive characterization of arbitrary evenly-good integers. 
This result provides a unified criterion for identifying such integers, clarifies their structural properties.

 \begin{theorem}[{\cite[Theorem 3.2]{J2018a}}]\label{thm-even} Let $a$ and $b$     be   coprime   nonzero  integers  and  let $\ell =2^\beta d$ be  an  integer such that $d$ is odd and $\beta \ge 0$. Then   one of the following statements holds.
    \begin{enumerate}
        \item  If  $ab$ is odd, then $\ell=2^\beta d\in EG_{(a,b)}$  if and only if  one of the following statements holds.
        \begin{enumerate}
            \item $\beta \in \{0,1\}$ and  $d=1$.
            \item $\beta \in \{0,1\}$, $d\geq 3$,  and    there exists $s\geq 2$ such that $2^s|| {\rm ord}_{p}(\frac{a}{b})$ for every prime $p$ dividing $d$.  
        \end{enumerate}
        \item  If $ab$ is even,  then $\ell=2^\beta d\in EG_{(a,b)}$  if and only if    one of the following statements holds.
        \begin{enumerate}
            \item  $\beta =0$ and  $d=1$.
            \item  $\beta =0$, $d\geq 3$,    and 
            there exists $s\geq 2$ such that $2^s|| {\rm ord}_{p}(\frac{a}{b})$ for every prime $p$ dividing $d$.  
        \end{enumerate}
    \end{enumerate} 
\end{theorem}

As a consequence, we obtain the following corollary; however, the converse does not necessarily hold. 
In contrast to oddly-good integers, the set of evenly-good odd integers does not   close   under the usual multiplication, which highlights a significant difference in their structural properties.

 \begin{corollary}[{\cite[Corollary 3.4]{J2018a}}]
 	\label{cor-evenly} Let $a$ and $b$ be coprime nonzero integers and let $d>1$ be an  odd  integer. If $d\in EG_{(a,b)}$   (resp. $d\in G_{(a,b)}$), then $j\in EG_{(a,b)}$   (resp. $j\in G_{(a,b)}$)  for all divisors $j$ of $d$. 
 \end{corollary}
From Theorem \ref{thm-even},  it can be deduced that 
         $4| {\rm ord}_{2^\beta  d}(\frac{a}{b})$  and   
$({a}{b}^{-1})^\frac{{\rm ord}_{2^\beta d}(\frac{a}{b})}{2}\equiv -1 \,({\rm mod}\, 2^\beta d)$   for all      $2^\beta d\in G_{(a,b)}$.  This  implies that   $\frac{{\rm ord}_{2^\beta d}(\frac{a}{b})}{2}$ is even and 
\[2^\beta d \big\vert \left(a^\frac{{\rm ord}_{2^\beta d}(\frac{a}{b})}{2} +b^\frac{{\rm ord}_{2^\beta d}(\frac{a}{b})}{2} \right)  .\]

 \begin{example}   Let  $a=5$ and $b=7$.    Then  $d= 11849= 17^2\cdot 41 \in EG_{(a,b)}$ with  $ {\rm ord}_{5}(\frac{a}{b}) =8$,  ${\rm ord}_{7}(\frac{a}{b})=40$. Precisely,  $2^3||{\rm ord}_{5}(\frac{a}{b}) $ and   $2^3||{\rm ord}_{13}(\frac{a}{b})$. Since ${\rm ord}_{d}(\frac{a}{b}) =680$, we have $d|(a^{340}+b^{340})$.

  By Theorem \ref{thm-even},  we have $2d=   23698 \in EG_{(a,b)}$.  Since ${\rm ord}_{2d}(\frac{a}{b}) ={\rm ord}_{d}(\frac{a}{b}) =680$ by Proposition \ref{prop2d}, we have $2d|(a^{340}+b^{340})$. Otherwise,   $2^\beta d \notin EG_{(a,b)}$ for all $\beta\geq 2$ by Theorem \ref{thm-even}.
 \end{example}

 	 \subsection{Use full Algorithm and Diagrams}

 	This subsection presents  two diagrams that visually illustrate the characterizations of evenly-good and oddly-good integers, as developed in the preceding discussions. The result highlighted in blue is adapted from \cite{M1997} (see Theorem~\ref{goodP}), while the remaining results  are summarized   from \cite{J2018b} and \cite{J2018a}. In accordance with Lemma~\ref{odd-good2}, the diagrams are separated into two distinct cases based on the parity of  $ab$. Specifically, the case where $ab$ is odd is presented in Figure~\ref{Fig3}, and the case where $ab$ is even  is given in Figure~\ref{Fig4}. To save space,  we write $OG$ and $EG$ instead of $OG_{(a,b)}$ and $EG_{(a,b)}$, respectively.

	\begin{figure}[!hbt]
\begin{tikzpicture} [
% auto,
% decision/.style = { diamond, draw=blue, thick, fill=blue!20,
% 	text width=5em, text badly centered,
% 	inner sep=1pt, rounded corners },
block/.style    = { rectangle, draw=black, 
    fill=black!5, text width=11em, text centered,
    rounded corners, minimum height=2em },
line/.style     = { draw, thick, ->, shorten >=2pt },
]
% Define nodes in a matrix
\matrix [column sep=2mm, row sep=8  mm] {
    
 &    \node [block] (ell) {$\ell$ };        &   & \\ %  
    
    \node [block] (ell1) {$1,2\in G\& OG\& EG$ };&&  &\node [block] (ell6) {$2^\beta d\in G$ iff $2^\beta,d\in G$ and $2||$ord$_d(ab^{-1}$) };\\
    
 \node [block] (ell3) {$2^\beta \in G$ iff $2^\beta|(a+b) $};  &\node [block] (ell4) {\color{blue} $d\in G$ iff $\exists s\geq 1, \forall p|d$, $2^s||$ord$_p(ab^{-1}$) };&\node [block] (ell5) {$2d\in G$ iff  $d\in G$ };  & \node [block] (OG3) {$2^\beta d\in G$ iff  $2^\beta d\in OG$};\\

     \node [block] (OG0) {$2^\beta \in G $ iff $2^\beta \in OG$}; && \node [block] (OG2) {$2d\in OG$ iff  $d\in OG$};&  \node [block] (EG2) {$2d\in EG$ iff  $d\in EG$ }; \\
    
     &   \node [block] (OG1) {$OG$ iff $\forall p|d$, $2||$ord$_p(ab^{-1}$)  };     &\node [block] (EG1) {$EG$ iff $\exists s\geq 2, \forall p|d$, $2^s||$ord$_p(ab^{-1}$)  }; & \\
};

\begin{scope} [every path/.style=line]
\path (ell)        --   node{\quad \quad $1,2$}   (ell1);  

%\path (ell)        --   node{\quad\quad $2$}   (ell2);  

\path (ell)        --   node{\quad\quad \quad   $2^\beta, ~\beta\geq 2$}   (ell3);  

\path (ell)        --   node{\quad \quad \quad \quad $d$ is odd.}   (ell4);  

\path (ell)        --   node{\quad \quad \quad \quad\quad $2d,~d$ is odd.}   (ell5);  

\path (ell)        --   node{\quad\quad \quad \quad\quad \quad \quad $2^\beta d,~d$ is odd and $\beta\geq 2$.}   (ell6);  

\path (ell6)        --      (OG3);  

\path (ell5)        --      (OG2);      
\path (ell5)        --       (EG2);   

\path (ell4)        --       (OG1);      
\path (ell4)        --      (EG1);   

\path (ell3)        --     (OG0);                      
\end{scope}
%
% legend for subprocedures
% legend for input and output variables
\end{tikzpicture}
\caption{Characterization of oddly-good and evenly-good integers with respect to coprime  integers   $a$ and $b$, where  $ab$ is   even.} \label{Fig3}
\end{figure}
 
	\begin{figure}[!hbt]
\begin{tikzpicture} [
% auto,
% decision/.style = { diamond, draw=blue, thick, fill=blue!20,
% 	text width=5em, text badly centered,
% 	inner sep=1pt, rounded corners },
block/.style    = { rectangle, draw=black, 
    fill=black!5, text width=11em, text centered,
    rounded corners, minimum height=2em },
line/.style     = { draw, thick, ->, shorten >=2pt },
]
% Define nodes in a matrix
\matrix [column sep=2mm, row sep=8  mm] {
    
    & \node [block] (ell) {$\ell$ };        &    &\node{$ab$ is even.};& \\ %  
    
    \node [block] (ell1) {$1$ is $G\& OG\& EG$ };&&  &\node [block] (ell6)  {$2^\beta d $ is bad };\\
    
    \node [block] (ell2)  {$2$ is bad};&  \node [block] (ell3)  {$2^\beta  $ is bad}; &\node [block] (ell4) { \color{blue} $d\in G$ iff $\exists s\geq 1, \forall p|d$, $2^s||$ord$_p(ab^{-1}$) };&\node [block] (ell5)  {$2d\ $  is bad};  &\\

    &&&& \\
    
    &   \node [block] (OG1) {$OG$ iff $\forall p|d$, $2||$ord$_p(ab^{-1}$)  };     &\node [block] (EG1) {$EG$ iff $\exists s\geq 2, \forall p|d$, $2^s||$ord$_p(ab^{-1}$)  }; &&& \\
};

\begin{scope} [every path/.style=line]
\path (ell)        --   node{\quad \quad $1$}   (ell1);  

\path (ell)        --   node{\quad\quad $2$}   (ell2);  

\path (ell)        --   node{\quad\quad \quad   $2^\beta, ~\beta\geq 2$}   (ell3);  

\path (ell)        --   node{\quad \quad \quad \quad $d$ is odd.}   (ell4);  

\path (ell)        --   node{\quad \quad \quad \quad\quad $2d,~d$ is odd.}   (ell5);  

\path (ell)        --   node{\quad\quad \quad \quad\quad \quad \quad $2^\beta d,~d$ is odd and $\beta\geq 2$.}   (ell6);

\path (ell4)        --       (OG1);      
\path (ell4)        --      (EG1);                
\end{scope}
%
% legend for subprocedures
% legend for input and output variables
\end{tikzpicture}
\caption{Characterization of oddly-good and evenly-good integers with respect to coprime  integers   $a$ and $b$, where  $ab$ is   even.} \label{Fig4}
\end{figure}

Algorithm~\ref{Alg2}   presents an efficient procedure  designed to identify and classify evenly-good and oddly-good integers based on the characterizations established in the previous sections. 
The algorithm provides a systematic procedure for determining whether a given integer belongs to either subclass and is derived directly from the results discussed earlier. 
Its step-by-step structure facilitates practical computation and enables a further  implementation.

\begin{algorithm}
	\caption{Identification  of the oddly and evenly-good  integers   with respect to coprime  integers  \(a\) and \(b\)}\label{Alg2}
	\begin{algorithmic}
		\Require Integers \(a\), \(b\), and \(\ell\geq 1\) such that $\gcd(a,b)=1$
		\If{\(\gcd(a,\ell) \ne 1\) or \(\gcd(b,\ell) \ne 1\)}
		\Return \(\ell \notin G_{(a,b)}\)  
		\Else
		\State Write \(\ell = 2^\beta d\) where \(d\) is odd and \(\beta \geq 0\)
		\If{\(d = 1\)}
		\If{\(2^\beta \nmid (a + b)\)}
		\Return \(\ell \notin G_{(a,b)} \) 
		\ElsIf{\(\beta = 0\)}
		\Return \(\ell \in OG_{(a,b)} \)  and    \(\ell \in EG_{(a,b)} \)    \quad %[\(\because\)  \(\ell |(a^1 + b^1)\) and  \(\ell |(a^2 + b^2)\).]
		\ElsIf{\(\beta = 1\) and \(ab\) is odd}
		\Return  \(\ell \in OG_{(a,b)} \)  and    \(\ell \in EG_{(a,b)} \) 
        %[\(\because\)  \(\ell |(a^1 + b^1)\) and  \(\ell |(a^2 + b^2)\).]
		\Else ~ \Return \(\ell \in OG_{(a,b)} \)   
        %[\(\because\)  \(\ell |(a^1 + b^1)\).]
		\EndIf
		\Else
		\State Factor \(d = p_1^{r_1}p_2^{r_2}\dots p_t^{r_t}\)
		\State Compute \(\text{ord}_{p_i}(ab^{-1})\) for \(i = 1, \dots, t\)
		\State Let \(S = \{s_i : 2^{s_i} \Vert \text{ord}_{p_i}(ab^{-1})\}\)
		\If{\(\beta \geq 1\) and \(ab\) is even}
		\Return \(\ell \notin G_{(a,b)}\)  
		\ElsIf{\(\beta = 0\)}
		\If{\(S = \{0\}\) or \(|S| \geq 2\)}
		\Return \(\ell \notin G_{(a,b)}\)  
		\ElsIf{\(S = \{1\}\)}
		\Return \(\ell \in OG_{(a,b)}\) 
		%\(\ell |\left(a^{\text{ord}_d(ab^{-1})/2} + b^{\text{ord}_d(ab^{-1})/2}\right)\)
		\Else ~
		 \Return \(\ell \in EG_{(a,b)}\)  
	%	\(\ell |\left(a^{\text{ord}_d(ab^{-1})/2} + b^{\text{ord}_d(ab^{-1})/2}\right)\)
		\EndIf
		\ElsIf{\(\beta = 1\) and \(ab\) is odd}
		\If{\(S = \{0\}\) or \(|S| \geq 2\)}
		\Return \(\ell \notin G_{(a,b)}\)  
		\ElsIf{\(S = \{1\}\)}
		\Return \(\ell \in OG_{(a,b)}\) 
		%\(\ell |\left(a^{\text{ord}_d(ab^{-1})/2} + b^{\text{ord}_d(ab^{-1})/2}\right)\)
		\Else ~
		\Return \(\ell \in EG_{(a,b)}\)  
	%	\(\ell |\left(a^{\text{ord}_d(ab^{-1})/2} + b^{\text{ord}_d(ab^{-1})/2}\right)\)
		\EndIf 
		\Else
		\If{\(S = \{1\}\) and \(2^\beta |(a + b)\)}
		\Return \(\ell \in OG_{(a,b)}\)  
		%\(\ell |\left(a^{\text{ord}_\ell(ab^{-1})/2} + b^{\text{ord}_\ell(ab^{-1})/2}\right)\)
		\Else ~
		\Return \(\ell \notin G_{(a,b)}\)  
		\EndIf
		\EndIf
		\EndIf
		\EndIf
	\end{algorithmic}
\end{algorithm}

Based on the procedure described in Algorithm~\ref{Alg2}, Tables~\ref{T2} and \ref{T3} present illustrative results for oddly-good and evenly-good integers, respectively. These tables visualize the output of the algorithm, providing concrete examples that reinforce the theoretical characterizations established in the preceding subsections.

\begin{table}[!hbt]
    \centering
    \begin{tabular}{ccl}
    \hline
       $a$  & $b$  & Oddly-good integers in $OG_{(a,b)}$\\
         \hline 
         $1$ & $2$ & $1,
3 ,
9 ,
11 ,
19 ,
27 ,
33 ,
43 , \dots$\\
$1$ & $3$ & $1, 2,
4 ,
7 ,
14 ,
19 ,
28 ,
31 ,
37 ,
38 ,
43 ,
49 , \dots$\\
$1$ & $4$ & $1,
5 ,
13 ,
25 ,
29 ,
37 ,
41 , \dots$\\
$1$ & $5$ & $1,2,
3 ,
6 ,
7 ,
9 ,
14 ,
18 ,
21 ,
23 ,
27 ,
29 ,
42 ,
43 ,
46 ,
47 ,
49 , \dots$\\
$1$ & $6$ & $1,
7 ,
11 ,
29 ,
31 ,
49 , \dots$
\\
$2$ & $3$ & $1,
5 ,
7 ,
11 ,
25 ,
31 ,
35 ,
49 , \dots$\\
$2$ & $5$ & $1,
7 ,
11 ,
19 ,
23 ,
37 ,
41 ,
47 ,
49 ,\dots$\\
$3$ & $4$ & $1,
7 ,
13 ,
19 ,
31 ,
43 ,
49 , \dots$\\
$3$ & $5$ & $1,2,
4 ,
8 ,
19 ,
23 ,
31 ,
38 ,
46 ,
47 , \dots$\\
$4$ & $5$ & $1,
3 ,
7 ,
9 ,
21 ,
23 ,
27 ,
43 ,
47 ,
49 , \dots$\\
$5$ & $6$ & $1,
11 ,
23 ,
31 ,
43 ,
47 ,\dots$\\
         \hline 
    \end{tabular}
    \caption{Oddly-good integers}\label{T2}
    \label{tab:placeholder}
\end{table}

\begin{table}[!hbt]
    \centering
    \begin{tabular}{ccl}
    \hline
       $a$  & $b$  & Evenly-good integers in $EG_{(a,b)}$\\
         \hline 
         $1$ & $2$ & $1,
5 ,
13 ,
17 ,
25 ,
29 ,
37 ,
41 ,\dots$\\
$1$ & $3$ & $1,2,
5 ,
10 ,
17 ,
25 ,
29 ,
34 ,
41 ,\dots $\\
$1$ & $4$ & $1,
17 ,\dots $\\
$1$ & $5$ & $ 1,2,
13 ,
17 ,
26 ,
34 ,
37 ,
41 , \dots$\\
$1$ &  $6$ & $1,
13 ,
17 ,
37 ,
41 , \dots$\\
$2$ & $3$ &  $1,
13 ,
17 ,
37 ,
41 , \dots$\\
$2$ & $5$ & $1,
17 ,
29 ,\dots$\\
$3$ &  $4$ & $1,
5 ,
17 ,
25 ,
29 ,
41 , \dots$\\
$3$ & $5$ & $1,2,
13 ,
17 ,
26 ,
29 ,
34 ,
37 ,
41 ,\dots $\\

$4$ & $5$ & $1,
13 ,
17 ,
37 ,
41 , \dots $\\
$5$ & $6$ & $1,
17 ,
41 , \dots$\\
         \hline 
    \end{tabular}
    \caption{Evenly-good integers}\label{T3}
    \label{tab:placeholder}
\end{table}

 \section{Applications} \label{sec4}
 
 This section presents a collection of applications of good integers and oddly-good integers in the context of coding theory. These applications arise in various structural studies of polynomials and cyclic codes over finite fields. In particular, we explore their roles in the characterization of self-reciprocal irreducible monic (SRIM) and self-conjugate-reciprocal  irreducible monic (SCRIM) factors of the polynomial $x^n - 1$ over finite fields, which are central to the algebraic analysis of cyclic codes. Furthermore, good integers  and oddly-good integers are applied to the construction,  classification, and enumeration of self-dual and complementary dual cyclic codes  over finite fields under both the  Euclidean and Hermitian inner products.

For the purpose of the applications discussed in this section, it is important to emphasize the characteristic of the underlying finite fields. Throughout this section, let $p$ be a prime. The notations $\mathbb{F}_{p^m}$ and $\mathbb{F}_{p^{2m}}$ are used to denote finite fields of characteristic $p$ with $p^m$ and $p^{2m}$ elements, respectively. These fields serve as the ambient spaces for the construction and analysis of cyclic codes, particularly when distinguishing between the  Euclidean inner product  and the Hermitian inner product.  Precisely, the following results   illustrate           applications of   $G_{(p^{m},1)}$ and   $OG_{(p^{m},1)}$ in coding theory, where $p$ is a prime and $m$ is a positive integer.

 \subsection{Cyclic Codes and Factorization of $x^n-1$ over Finite Fields}

In this subsection, we provide a concise overview of the factorization of the polynomial $x^n - 1$ over a finite field $\mathbb{F}_{p^m}$, which plays a central role in the study of cyclic codes of length $n$ over $\mathbb{F}_{p^m}$. We also examine the algebraic structure of cyclic codes as ideals in the quotient ring $\mathbb{F}_{p^m}[x]/\langle x^n - 1 \rangle$. These foundational concepts are essential for understanding the properties of specific families of polynomials and cyclic codes that will be explored in the subsequent subsections.

\subsubsection{Factorization of $x^n-1$ over Finite Fields}

  Let $p$ be a prime and let $n$ be a positive integer and write $n=Np^t$, where $N$ is a positive integer such that $p\nmid N$ and $t\geq 0$ is an integer.  For an element $a\in \mathbb{Z}_{N}$,  let ${\rm ord}(a)$  denote the additive order of $a$ in $\mathbb{Z}_{N}$.   For  $a\in \mathbb{Z}_{N}$, the $p^m$-cyclotomic coset of $\mathbb{Z}_{N}$ containing $a$ is defined to be the set 
\begin{align}
C_q(a)=\{p^{mi}a \mid i\in \{0,1,2,\dots\}  \},
\end{align}
 where  $p^{mi}a =\sum_{j=0}^{p^{mi}} a \in \mathbb{Z}_{N}$.  Alternatively,  we may view $p^{mi}a $ as an integer modulo $N$.  It is not difficult to see that  $|C_{p^{m}}(a)|= {\rm ord}_{\rm ord(a)}(p^{m})$. For each divisor $d$ of $N$, let 
\begin{align}
    A_d=\{a\in \mathbb{Z}_{N} \mid {\rm ord}(a)=d \}, 
\end{align} where the union is not necessarily  disjoint. 
Then  $|A_d|= \Phi(d)$ and the elements in $A_d$ are partitioned in to $\dfrac{\Phi(d)}{{\rm ord}_{d}(p^{m})}$  \ \ $p^{m}$-cyclotomic cosets  of the same order ${\rm ord}_{d}(q)$     (see \cite[Remark 2.5]{JLS2014}), where $\Phi$ is the Euler phi function.  It follows that   \[A_d= \bigcup_{\substack{a\in A\\  {\rm ord}(a)=d}} C_q(a).\]
For each divisor $d$ of $N$, let $\{a_{d1},  a_{d2},  \dots, a_{d\gamma_d}\}$ be a complete set of representative of the $p^m$-cyclotomic cosets in $A_d$, where    $\gamma_d = \dfrac{\Phi(d)}{{\rm ord}_{d}(p^{m})}$. 
Then  we have the following disjoint union
\begin{align}
     \mathbb{Z}_{N}=\bigcup_{d|N} A_d=\bigcup_{d|N} \bigcup _{j=1}^{\gamma_d} C_
     {p^{m}}(a_{dj}).
\end{align}

In \cite[Theorem 3.4.8 and   Theorem 3.4.11]{LS2004},  the factorization of $x^N-1$ over $\mathbb{F}_{p^m}$ can be given in terms of $p^m$-cyclotomic cosets  of  $\mathbb{Z}_N$ (cf. \cite[Equation 3]{SJLP2015}.   Let  $\alpha$ be a primitive $N$th root of unity.  For each  $a \in \mathbb{Z}_N$,  let \[f_a(x)=\prod_{i\in C_{p^{m}}(a)} (x-\alpha^i)\] be the monic polynomial determined by $C_{p^{m}}(a)$.  From \cite[Theorem 3.4.11]{LS2004}, it follows that  $f_a(x)$ is irreducible and it is the minimal polynomial of $\alpha$ and 
\begin{align} \label{facXN-1}
    x^N-1=  \prod_{d|N} \prod _{j=1}^{\gamma_d}  f_{a_{dj}} (x)
\end{align}
which implies that 
\begin{align}\label{facXn-1}
  x^n-1=  x^{Np^t}-1= (x^{N}-1) ^{p^t} =\prod_{d|N} \prod _{j=1}^{\gamma_d}  f_{a_{dj}} (x) ^{p^t}.
\end{align}

\begin{example} \label{ex5} The factorization of $x^{60}-1$ over $\mathbb{F}_3$ is given as follows. First, we note that $n=60=20\cdot 3$ and the $3$-cyclotomic cosets of $\mathbb{Z}_{20}$ are 
    $   \{ 0 \},  
    \{ 1, 3, 7, 9 \},
     \{ 2, 6, 14, 18 \},
      \{  4, 8, 12, 16 \},
    \{ 5, 15 \},
      \{ 10 \}, $ and $
    \{ 11, 13, 17, 19 \}$. Let $\alpha$ be a primitive $20$th root of unity in $\mathbb{F}_{3^4}$. Then $f_0(x)=x + 2, f_1(x)= x^4 + x^3 + 2x + 1, f_2(x)=x^4 + 2x^3 + x^2 + 2x + 1,  f_4(x)= x^4 + x^3 + x^2 + x + 1, f_5(x)=x^2 + 1, f_{10}(x)=x + 1, $ and $f_{11}(x)=x^4 + 2x^3 + x + 1$. It follows that 
    \[ x^{20}-1= (x + 2)(x^4 + x^3 + 2x + 1)(x^4 + 2x^3 + x^2 + 2x + 1)( x^4 + x^3 + x^2 + x + 1)(x^2 + 1)(x + 1)(x^4 + 2x^3 + x + 1) \]
    and 
\[ x^{60}-1=(x^{20}-1)^3= (x + 2)^3(x^4 + x^3 + 2x + 1)^3(x^4 + 2x^3 + x^2 + 2x + 1)^3( x^4 + x^3 + x^2 + x + 1)^3(x^2 + 1)^3(x + 1)^3(x^4 + 2x^3 + x + 1)^3 .\]

\end{example}
\subsubsection{Cyclic Codes}

For integers  $n$ and $0\leq k \leq n$,    a \textit{linear code} $C$ of length $n$ and dimension $k$ over $\mathbb{F}_{p^{m}}$ is defined as a $k$-dimensional subspace of  the $\mathbb{F}_{p^{m}}$-vector space $\mathbb{F}_{p^{m}}^n$. In this case,  $C$ is  referred to as an $[n,k]_{p^{m}}$ code.  %The error-detecting and error-correcting capabilities of a linear code are determined by its minimum Hamming weight. If 
%\[\mathrm{wt}(C) = \min\{ \mathrm{wt}(\boldsymbol{c}) \mid \boldsymbol{c} \in C \setminus \{\boldsymbol{0}\} \} = d,\]
%then $C$ is said to be an $[n,k,d]_q$ code, where $d$ denotes the minimum weight. The Hamming weight of a vector $\boldsymbol{x} = (x_1, x_2, \dots, x_n) \in \mathbb{F}_q^n$ is defined as
%\[ \mathrm{wt}(\boldsymbol{x}) = |\{ i \in \{1,2,\dots,n\} \mid x_i \ne 0 \}|.\]
The \textit{Euclidean dual} of a linear code $C$ of length $n$ over $\mathbb{F}_{p^{m}}$ is defined to be the set
\[
C^{\perp_{E}} = \left\{ \boldsymbol{x}  \in \mathbb{F}_{p^{m}}^n \,\middle\vert\, \langle \boldsymbol{x}, \boldsymbol{c} \rangle_{E} =0 \text{ for all } \boldsymbol{c} \in C \right\},
\]
where $\langle \boldsymbol{x}, \boldsymbol{y} \rangle_{E}  =\sum_{i=0}^{n-1} x_i y_i $ is the {\em Euclidean inner product} of $ \boldsymbol{x}=
(x_0,\ldots,x_{n-1}),  
\boldsymbol{y}=
(y_0,\ldots,y_{n-1}) \in \mathbb{F}_{p^m}^n$.
It is easy to verify that $C^{\perp_{E}}$ is again a linear code of length $n$ over $\mathbb{F}_{p^{m}}$ and it satisfies the  relation $\dim(C) + \dim(C^{\perp_{E}}) = n$.
A linear code $C$ is said to be \textit{Euclidean self-orthogonal} if $C \subseteq C^{\perp_{E}}$  and  it is called a \textit{Euclidean  self-dual} code if  $C = C^{\perp_{E}}$. The code $C$ is called {\em Euclidean  complementary dual} if  $C \cap C^{\perp_{E}} =\{0\}$. In the same fashion,   the \textit{Hermitian dual} of a linear code $C$ of length $n$ over $\mathbb{F}_{p^{2m}}$ is defined to be the set
\[
C^{\perp_{H}} = \left\{ \boldsymbol{x}  \in \mathbb{F}_{p^{2m}}^n \,\middle\vert\, \langle \boldsymbol{x}, \boldsymbol{c} \rangle_{H} =0 \text{ for all } \boldsymbol{c} \in C \right\},
\]
where $\langle \boldsymbol{x}, \boldsymbol{y} \rangle_{H}  =\sum_{i=0}^{n-1} x_i y_i^{p^m} $ is the {\em Hermitian  inner product} of $ \boldsymbol{x}=
(x_0,\ldots,x_{n-1}),  
\boldsymbol{y}=
(y_0,\ldots,y_{n-1}) \in \mathbb{F}_{p^{2m}}^n$.
The Hermitian dual  $C^{\perp_{H}}$ of a linear code $C$ is    linear   and   $\dim(C) + \dim(C^{\perp_{H}}) = n$.
A linear code $C$ is said to be \textit{Hermitian self-orthogonal} if $C \subseteq C^{\perp_{H}}$ and it is called a \textit{Hermitian self-dual} code if  $C = C^{\perp_{H}}$. The code $C$ is called {\em Hermitian   complementary dual} if  $C \cap C^{\perp_{H}} =\{0\}$. Self-dual linear codes and linear complementary dual   codes over finite fields have important practical applications in modern communication systems and are closely connected to various mathematical structures (see \cite{H2021},  \cite{H2010}, and references therein).  
Their properties provide useful tools for error detection and correction, and they also establish links with other algebraic objects studied in coding theory.

A linear code $C$ of length $n$ over $\mathbb{F}_{p^m}$ is said to be \textit{cyclic} if it is closed under the cyclic shift. Precisely,  $(c_{n-1}, c_0, \ldots, c_{n-2})\in C$ for all  $(c_0, c_1, \ldots, c_{n-1}) \in C$. Cyclic codes form an important subclass of linear codes that have been extensively investigated due to their rich algebraic structure and wide range of theoretical and practical applications (see \cite{H2021} and \cite{H2010}). 
One of their notable advantages is that they can be efficiently encoded and decoded using shift registers, which makes them particularly attractive for implementation in communication and storage systems.
Subsequently, cyclic codes can be analyzed effectively using tools from polynomial algebra, thereby allowing algebraic techniques to be applied directly in their study.
Precisely,     every cyclic code of length $n$ over $\mathbb{F}_{p^m}$ corresponds to an ideal in the quotient ring $\mathbb{F}_{p^m}[x]/\langle x^n - 1 \rangle$ via the map $\pi: \mathbb{F}_{p^m}^n \to \mathbb{F}_q[x]/\langle x^n - 1 \rangle$ defined by $(a_0,a_1,\dots, a_{n-1}\mapsto a_0+a_1x+\dots+a_{n-1}x^{n-1}$ (see  \cite{H2010}). Each such ideal is principal and generated by a unique monic polynomial $g(x)$  which is a divisor of  $x^n - 1$. This polynomial is referred to as the \textit{generator polynomial} of the cyclic code $C$, and it encapsulates all the algebraic structure of $C$.  The following  well-known   theorems (see, e.g.,   \cite{LS2004} and \cite{H2010})  are useful in the study of the algebraic structures  of cyclic codes.

\begin{theorem} Let $p$ be a prime and let $n$ be a positive integer and write $n=Np^t$, where $N$ is a positive integer such that $p\nmid N$ and $t\geq 0$ is an integer.  Assume that $x^n-1$ has  the  factorization as in 
\eqref{facXn-1}.  Then each cyclic code of length $n$ over $\mathbb{F}_{p^m}$ has the generator polynomial of the form
\[ g(x)=    \prod_{d|N} \prod _{j=1}^{\gamma_d}  f_{a_{dj}} (x) ^{\ell_{d_j}},\]
where $0\leq \ell_{d_j }<p^t$. 
\end{theorem}

\begin{theorem} \label{genC}
Let $C$ be a cyclic code of length $n$ over $\mathbb{F}_{p^m}$  with the generator  polynomial $g(x)$, and let $k$ be an integer such that $0 \leq k \leq n$. Then $C$ has dimension $k$ if and only if $\deg(g(x)) = n - k$.
\end{theorem}

 \subsection{SRIM   Factors of $x^n-1$ over Finite Fields and Applications}

The reciprocal polynomial of a  polynomial  $f(x)=\sum_{i=0}^kf_k x^k$  of degree $k$ in $ \mathbb{F}_{p^m}[x]$ with $f_0\ne 0$  is defined to be $f^*(x)= f_0^{-1} x^k f(\frac{1}{x})$, and it is said to be {\em self-reciprocal irreducible monic} (SRIM) if  $f(x)$ is irreducible, $f(x)=f^*(x)$, and $f_k=1$. Otherwise,  $f(x)$ and $f^*(x)$ are called a {\em reciprocal polynomial pair}. The factorization and the SRIM factors of $x^n-1$ over $\mathbb{F}_{p^m}$  are key to study  Euclidean complementary dual cyclic codes and Euclidean self-dual cyclic codes of length $n$ over  $\mathbb{F}_{p^m}$.  The details are presented in the following subsections.

 \subsubsection{SRIM   Factors of $x^n-1$ over Finite Fields}

For an element $a \in \mathbb{Z}_N$, the ${p^m}$-cyclotomic coset $C_{p^m}(a)$   is said to be of \emph{type} $I$ if 
$
C_{p^m}(a) = C_{p^m}(-a),
$
and of \emph{type} $II$ otherwise.  
According to   \cite[Lemma 3.4]{BJ2021-s},  \cite[Section 4.1]{CLZ2016}, \cite[Lemma~4.8]{JLLX2012}, and \cite[Lemma~3]{SJLP2015}, the complete characterization of the  SRIM  factors of the polynomial $x^n - 1$ over the finite field $\mathbb{F}_{p^m}$ can be expressed in terms of ${p^m}$-cyclotomic cosets of $\mathbb{Z}_N$. In particular, SRIM factors correspond precisely to those cyclotomic cosets of  type $I$, while type~$II$ cyclotomic  cosets generate pairs of reciprocal irreducible factors.  The results are summarized as follows.

\begin{lemma} \label{lem:f*} Let $p$ be a prime power and let $N$ be a positive integer  such that $p\nmid N$.  Let $a\in \mathbb{Z}_N$.  Then  $f(x)$ is induced by  $C_{p^m}(a) $  if and only if   $f^*(x)$ is induced by  $C_{p^m}(-a) $. 
    
\end{lemma}

\begin{lemma} \label{charFacSrim} Let $p$ be a prime power and let $N$ be a positive integer  such that $p\nmid N$.   Let $a\in \mathbb{Z}_N$   and let $f_a(x)$ be  the monic irreducible polynomial  over $\mathbb{F}_{p^{m}}$ induced by  $C_{p^m}(a)$.  Then the following statements are equivalent.
\begin{enumerate}
    \item  $f_a(x)$ is  a SRIM factor of $x^N-1$.
    \item $C_{p^m}(a) $  is of type $I$.
    \item ${\rm ord}(a) \in G_{({p^m},1)}$.
\end{enumerate}
\end{lemma}

Based on the rearrangement of the factors in \eqref{facXN-1} and the characterization in Lemma \ref{charFacSrim},  we have the following factorization which is useful for the further  study of  Euclidean self-dual  and Euclidean complementary dual cyclic codes over $\mathbb{F}_{p^m}$.

\begin{align} \label{refacXN-1}
    x^N-1=  \prod_{\substack{d|N\\ d\in G_{({p^m},1)}}} \prod _{j=1}^{\gamma_d}  f_{a_{dj}} (x) \prod_{\substack{d|N\\ d\notin G_{({p^m},1)}}} \prod _{j=1}^{\gamma_d/2}  f_{a_{dj}} (x) f_{a_{dj}} ^*(x),
\end{align}
where  $f_{a_{dj}} ^*(x)= f_{-a_{dj}} (x)$ by Lemma \ref{lem:f*}. Consequently, if $n={Np^t}$, we have 
\begin{align}\label{refacXn-1}
  x^n-1=  x^{Np^t}-1=  \prod_{\substack{d|N\\ d\in G_{({p^m},1)}}} \prod _{j=1}^{\gamma_d}  f_{a_{dj}} (x) ^{p^t} \prod_{\substack{d|N\\ d\notin G_{({p^m},1)}}} \prod _{j=1}^{\gamma_d/2}  f_{a_{dj}} (x)^{p^t} f_{a_{dj}} ^*(x) ^{p^t}.
\end{align}

\begin{example} \label{ex6}
    From Example \ref{ex5},  the $3$-cyclotomic cosets of $\mathbb{Z}_{20}$ are 
    $   \{ 0 \},  
    \{ 1, 3, 7, 9 \},
     \{ 2, 6, 14, 18 \},
      \{  4, 8, 12, 16 \},
    \{ 5, 15 \},
      \{ 10 \}, $ and $
    \{ 11, 13, 17, 19 \}$.  
    Then    ${\rm ord}(0)=1 \in G_{(3,1)},
    {\rm ord}(1)=20 \notin G_{(3,1)}, 
    {\rm ord}(2)=10 \in G_{(3,1)},
    {\rm ord}(4)=5 \in G_{(3,1)},
    {\rm ord}(5)=4 \in G_{(3,1)},
    {\rm ord}(10)=2 \in G_{(3,1)},$ and $ {\rm ord}(11)=20 \notin G_{(3,1)}$. It follows that 
      $f_0(x)=x + 2, f_2(x)=x^4 + 2x^3 + x^2 + 2x + 1,  f_4(x)= x^4 + x^3 + x^2 + x + 1, f_5(x)=x^2 + 1, f_{10}(x)=x + 1 $ are SRIM,  but    $f_1(x)= x^4 + x^3 + 2x + 1$ and $f_{11}(x)=x^4 + 2x^3 + x + 1$ form a reciprocal polynomial pair. 
      It follows that 
    \[ x^{20}-1= (x + 2)(x^4 + 2x^3 + x^2 + 2x + 1)( x^4 + x^3 + x^2 + x + 1)(x^2 + 1)(x + 1)  \left((x^4 + x^3 + 2x + 1)(x^4 + 2x^3 + x + 1)\right) \]
    and 
\[ x^{60}-1 = (x + 2)^3(x^4 + 2x^3 + x^2 + 2x + 1)^3( x^4 + x^3 + x^2 + x + 1)^3(x^2 + 1)^3(x + 1)^3 \left((x^4 + x^3 + 2x + 1)^3(x^4 + 2x^3 + x + 1)^3 \right),\] where the first five terms are SRIM and  the last two terms form a reciprocal polynomial pair.
\end{example}

   \subsubsection{Euclidean Complementary Dual Cyclic Codes over Finite Fields}
\label{secECD}
In \cite{YM1994}, complementary dual  cyclic codes defined under the  Euclidean  inner product over finite fields were first introduced and studied. 
Such codes are  defined to be  cyclic codes  $C$ whose intersection with their Euclidean dual is trivial, i.e., 
$ C \cap C^{\perp_E} = \{ {0}\} $.
These codes possess elegant algebraic structures and have found wide applications in various fields, including adder-channel communications and the construction of quantum error-correcting codes.

Here, the algebraic structure, characterization, and construction of  such  cyclic codes  are revisited by employing the SRIM  factors of the polynomial $x^N - 1$ together with the framework of good integers.  
Furthermore, the explicit number of    Euclidean complementary dual cyclic codes over finite fields is presented.

\begin{theorem}[{\cite[Theorem 7.3.7]{LS2004}}] \label{genED}
Let $C$ be a cyclic code of length $n$ over $\mathbb{F}_{p^m}$  with the generator  polynomial $g(x)$.  Then the  Euclidean dual $C^{\perp_{E}}$ of $C$   is $h^*(x)$, where $h(x)= \dfrac{x^n-1}{g(x)}$.
\end{theorem}

 Since $C\cap C^{\perp_E} $ is a cyclic code  generated by  $LCM(g(x),h^*(x)$, $C$ is Euclidean complementary dual if and only if    $LCM(g(x),h^*(x)) = x^n-1$, or equivalently,  $GCD(g(x),h^*(x)) = 1$.  Hence, the next theorem follows.

\begin{theorem}[{\cite[Theorem]{YM1994}}]  Let $p$ be a prime and let $n$ be a positive integer and write $n=Np^t$, where $N$ is a positive integer such that $p\nmid N$ and $t\geq 0$ is an integer. 
    Let $C$ be a cyclic code of length $n$ over $\mathbb{F}_{p^m}$ with generator polynomial $g(x)$. Then $C$ is Euclidean complementary dual if and only  if $g(x)$  is self-reciprocal and each irreducible divisor has multiplicity $p^t$.
\end{theorem}

In the case where $t=1$,  we have $\gcd(N,p^m)=1$  and the next corollary follows immediately.  
\begin{corollary}   
Let $p$ be a prime  and let $N$ be a positive integer such that $p\nmid N$. 
    Let $C$ be a cyclic code of length $N$ over $\mathbb{F}_{p^m}$ with generator polynomial $g(x)$. Then $C$ is Euclidean complementary dual if and only  if $g(x)$  is self-reciprocal.
\end{corollary}

The following theorem are direct consequence of Theorem \ref{genED}. Alternatively, they can be viewed as a lighter version of \cite[Corollary 4 and Proposition 5]{BJU2018}. 

\begin{theorem}\label{thm9}  Let $p$ be a prime and let $n$ be a positive integer and write $n=Np^t$, where $N$ is a positive integer such that $p\nmid N$ and $t\geq 0$ is an integer.   Assume the factorization in \eqref{refacXn-1} and   let $C$ be a cyclic code of length $n$ over $\mathbb{F}_{p^m}$ with generator polynomial 
\[ g(x) =  \prod_{\substack{d|N\\ d\in G_{({p^m},1)}}} \prod _{j=1}^{\gamma_d}  f_{a_{dj}} (x) ^{r_{dj} } \prod_{\substack{d|N\\ d\notin G_{({p^m},1)}}} \prod _{j=1}^{\gamma_d/2}  f_{a_{dj}} (x)^{s_{dj}} f_{a_{dj}} ^*(x) ^{s'_{dj}} .\]
    Then $C$ is Euclidean complementary dual if and only   if $ r_{dj} \in \{0,p^t\}$ for all  $d|N$ such that $d\in G_{({p^m},1)}$ and  $j=1,2,\dots, \gamma_d$, and $( s_{dj},s'_{dj}) \in \{(0,0),(p^t,p^t)\}$  for all  $d|N$ such that $d\notin G_{({p^m},1)}$ and  $j=1,2,\dots, \gamma_d/2$.
\end{theorem}

\begin{theorem} \label{thm10}
Let $p$ be a prime and let $n$ be a positive integer and write $n=Np^t$, where $N$ is a positive integer such that $p\nmid N$ and $t\geq 0$ is an integer. Then the number of Euclidean  complementary dual codes of length $n$ over $\mathbb{F}_{p^m}$  is 
    \[ 2^ {\displaystyle \sum_{\substack{d|N\\ d\in G_{({p^m},1)}}} \frac{\Phi(d)}{{\rm ord}_d (p^m)} + \sum_{\substack{d|N\\ d\notin G_{({p^m},1)}}}  \frac{\Phi(d)}{2{\rm ord}_d (p^m)} }.\]
\end{theorem}

\begin{example}
    From Example \ref{ex6},   we have the following factorization for  $x^{60}-1$ over $\mathbb{F}_3$:
\[ x^{60}-1 = (x + 2)^3(x^4 + 2x^3 + x^2 + 2x + 1)^3( x^4 + x^3 + x^2 + x + 1)^3(x^2 + 1)^3(x + 1)^3 \left((x^4 + x^3 + 2x + 1)^3(x^4 + 2x^3 + x + 1)^3 \right) ,\] where the first four terms are SRIM and  the last two terms form a reciprocal polynomial pair. Using Theorem \ref{thm9}, the following polynomials are generator polynomials  for Euclidean complementary dual  cyclic codes of length $60$ over $\mathbb{F}_3$: 
$(x + 2)^3(x^4 + 2x^3 + x^2 + 2x + 1)^3 $, $ (x^2 + 1)^3(x + 1)^3 \left((x^4 + x^3 + 2x + 1)^3(x^4 + 2x^3 + x + 1)^3 \right)$,  and $ (x + 2)^3 \left((x^4 + x^3 + 2x + 1)^3(x^4 + 2x^3 + x + 1)^3 \right)$. 
Based on Theorem  \ref{thm10}, the number of Euclidean complementary dual  cyclic codes of length $60$ over $\mathbb{F}_3$ is
\[ 2^ {\displaystyle \sum_{\substack{d|20\\ d\in G_{({3},1)}}} \frac{\Phi(d)}{{\rm ord}_d (3)} + \sum_{\substack{d|20\\ d\notin G_{({3},1)}}}  \frac{\Phi(d)}{2{\rm ord}_d (3)} }=2^{\displaystyle \sum_{d\in \{1,2,4,5,10\}} \frac{\Phi(d)}{{\rm ord }_d(3)}+\sum_{d\in \{20\}}  \frac{\Phi(d)}{2{\rm ord }_d(3)} } =2^{(1+1+1+1+1)+ 1}=2^6 .\]

\end{example}

 \subsubsection{Euclidean Self-Dual Cyclic Codes over Finite Fields}

A cyclic code $C$ of length $n$ over a finite field $\mathbb{F}_{p^m}$ is called \emph{Euclidean self-dual} if it coincides with its Euclidean dual, that is, $C = C^{\perp_E}$.  
The class of self-dual cyclic codes represents one of the most significant families of linear codes, not only because of their elegant algebraic structure but also due to their importance in practical applications such as data transmission and error correction.  
Owing to these features, this family of codes has been the subject of continuous research for several decades.  
In particular, a complete characterization together with explicit enumeration of Euclidean self-dual cyclic codes over finite fields was established in \cite{JLX2011}.  

To set the stage for further discussion, recall that if $C$ is a cyclic code of length $n$ over $\mathbb{F}_{p^m}$ with generator polynomial $g(x)$, then $C$ is Euclidean self-dual precisely when 
$
   g(x) = h^*(x),
$
where $h(x) = \frac{x^n-1}{g(x)}$ and $h^*(x)$ denotes the reciprocal polynomial of $h(x)$.  
This identity imposes strong restrictions on the parameters of the code: in fact, it can occur only in the case where both the length $n$ is even and the   underlying field has characteristic $2$.  
The precise statement of this condition is formalized in the following theorem.

\begin{theorem}[{\cite[Theorem 1]{JLX2011}}]
  \label{thm11} Let $p$ be a prime and let $n$ be a positive integer and write $n=Np^t$, where $N$ is a positive integer such that $p\nmid N$ and $t\geq 0$ is an integer. Then there exists a Euclidean self-dual cyclic code of length $n$ over $\mathbb{F}_{p^m}$ if and only if $p=2$ and $t\geq 1$. 
 \end{theorem}
In this case, we have that $n=2^tN$ is even for some  odd positive integer $N$  and positive integer $t$, and the underlying field is $\mathbb{F}_{2^m}$ for some positive integer $m$.  Based on Theorem   \ref{thm11} and the fact that a cyclic code of length $n$ over $\mathbb{F}_{2^m}$ with generator polynomial $g(x)$ is Euclidean self-dual if and only if $g(x)=h^*(x)$, the next theorems follows.  

\begin{theorem}[{\cite[Theorem 2]{JLX2011}}]
\label{thm12}
  Let $n$ be an even  positive integer and write $n=N2^t$, where $N$ is an odd positive integer  and $t$ is a positive integer.
Assume the factorization in \eqref{refacXn-1}  and let $C$ be  a cyclic code  of length $n$ over $\mathbb{F}_{2^m}$ with generator  polynomial $g(x)$.  Then $C$ is Euclidean self-dual if and only if 
  \[ g(x) =  \prod_{\substack{d|N\\ d\in G_{({2^m},1}}} \prod _{j=1}^{\gamma_d}  f_{a_{dj}} (x) ^{2^{t-1} } \prod_{\substack{d|N\\ d\notin G_{({2^m},1)}}} \prod _{j=1}^{\gamma_d/2}  f_{a_{dj}} (x)^{s_{dj}} f_{a_{dj}} ^*(x) ^{p^t-s_{dj}} ,\] where $0\leq  {s_{dj}} \leq p^t$ for all   $d|N$ such that $d\notin G_{({2^m},1)}$ and  $j=1,2,\dots, \gamma_d/2$.
\end{theorem}

\begin{theorem}[{\cite[Corollary 1 and Theorem 3]{JLX2011}}] \label{thm13}
  Let $n$ be an even  positive integer and write $n=N2^t$, where $N$ is an odd positive integer  and $t$ is a positive integer.
Then the number of  Euclidean self-dual  cyclic codes  of length $n$ over $\mathbb{F}_{2^m}$ is
\[(1+2^t) ^ {\displaystyle \sum _{ {\substack{d|N\\ d\notin G_{({2^m},1)}}}  }  \frac{\Phi(d)}{2 {\rm ord}_d(2^m)}}.\]
\end{theorem}

\begin{example}  The $2$-cyclotomic cosets of $\mathbb{Z}_{15}$ are presented in Table \ref{T4} together with their corresponding polynomials over $\mathbb{F}_2$. 

\begin{table}[!hbt]
    \centering
    \begin{tabular}{ccccc}
    \hline 
        $C_2(a)$  &  ${\rm ord}(a) $ &${\rm ord}(a) \in G_{(2,1)}$ & $f_a(x)$ & type \\
        \hline 
        $\{ 0 \}$ &$1$ & yes & $ x + 1$ & SRIM\\
$\{ 1, 2, 4, 8 \}$ &  $15$& no &$
 x^4 + x + 1$ & - \\

$\{  3, 6, 9,12 \}$& $5$& yes & $
 x^4 + x^3 + x^2 + x + 1 $ & SRIM \\

$\{ 5, 10 \}$ & $3$ & yes  & $ x^2 + x + 1$ & SRIM \\

$\{ 7, 11, 13, 14 \}$ & $15$& no &$
 x^4 + x^3 + 1$ & - \\
\hline
    \end{tabular}
    \caption{$2$-cyclotomic cosets in $\mathbb{Z}_{15}$ and their corresponding polynomials over $\mathbb{F}_{2}$}
    \label{T4}
\end{table}
The factorization of $x^{15}-1$ over $\mathbb{F}_{2}$ is of the form 
\[ x^{15}-1=
( x + 1)( x^2 + x + 1)(x^4 + x^3 + x^2 + x + 1)\left(( x^4 + x + 1 )( x^4 + x^3 + 1)\right)
\]
and
\[ x^{30}-1= (x^{15}-1)^2=
( x + 1)^2( x^2 + x + 1)^2(x^4 + x^3 + x^2 + x + 1)^2\left(( x^4 + x + 1 )^2( x^4 + x^3 + 1)^2\right),
\]
where the first three terms are SRIM and the last two terms form a reciprocal polynomial pair.  From Theorem \ref{thm12}, the   Euclidean self-dual  cyclic codes  of length $30$ over $\mathbb{F}_{2}$ are generated by   $( x + 1)( x^2 + x + 1)(x^4 + x^3 + x^2 + x + 1)( x^4 + x + 1 )^2 $, $( x + 1)( x^2 + x + 1)(x^4 + x^3 + x^2 + x + 1) ( x^4 + x^3 + 1)^2$, and $( x + 1)( x^2 + x + 1)(x^4 + x^3 + x^2 + x + 1)\left(( x^4 + x + 1 )( x^4 + x^3 + 1)\right)$.
Based on Theorem \ref{thm13}, the number of  Euclidean self-dual  cyclic codes  of length $30$ over $\mathbb{F}_{2}$ is
\[(1+2) ^ {\displaystyle \sum _{ {\substack{d|15\\ d\notin G_{({2},1)}}}  }  \frac{\Phi(d)}{2 {\rm ord}_d(2)}} =  (1+2) ^ {\displaystyle \sum _{d \in \{ 15\} }  \frac{\Phi(d)}{2 {\rm ord}_d(2)}} = 3^1=3.\] 
    
\end{example}

 \subsection{SCRIM Factors of $x^n-1$ over  Finite Fields and Applications}

Let $f(x)=\sum_{i=0}^k f_i x^i$ be a polynomial of degree $k$ in $\mathbb{F}_{p^{2m}}[x]$ with constant term $f_0 \neq 0$.  
The \emph{conjugate-reciprocal polynomial} of $f(x)$ is defined by  
\[
   f^\dagger(x) = f_0^{-p^m} x^k \sum_{i=0}^k f_i^{p^m} x^{-i}.
\]  
A polynomial $f(x)$ is said to be {self-conjugate-reciprocal irreducible monic} ({SCRIM}) if it satisfies the following conditions:  
 $f(x)$ is irreducible over $\mathbb{F}_{p^{2m}}$,  
$f(x) = f^\dagger(x)$, and  
 the leading coefficient $f_k = 1$.  
If these conditions do not hold, then $f(x)$ and $f^\dagger(x)$ together form   a \emph{conjugate-reciprocal polynomial pair}.  

In the subsequent subsections, SCRIM factors of $x^n - 1$ over $\mathbb{F}_{p^{2m}}$ will play a central role.  
In particular, they are crucial in the analysis and classification of Hermitian complementary dual cyclic codes as well as Hermitian self-dual cyclic codes of length $n$ over $\mathbb{F}_{p^{2m}}$.

 \subsubsection{SCRIM   Factors of $x^n-1$ over Finite Fields} 

   For an element  $a\in \mathbb{Z}_N$, the ${p^{2m}}$-cyclotomic coset $C_{p^{2m}}(a) $ is said to be of type $I'$ if   $C_{p^{2m}}(a)=C_{p^{2m}}(-p^ma)  $,  and it is of type $II'$ otherwise.   
According to~\cite[Section~2.1]{BJU2019}, the factorization of $x^n - 1$ over the finite field $\mathbb{F}_{p^{2m}}$ and  the complete characterization of its SCRIM factors are  described in terms of the ${p^{2m}}$-cyclotomic cosets  of $\mathbb{Z}_N$.  
In particular, the SCRIM factors correspond exactly to those cyclotomic cosets classified as type~${I}'$, whereas the cosets of type~${II}'$ generate pairs of conjugate-reciprocal   factors.   This characterization has a direct  link with oddly-good integers in $OG_{(p^m,1)}$.
The key results are summarized as follows.

\begin{lemma} \label{lem:fdagg} Let $p$ be a prime power and let $N$ be a positive integer  such that $p\nmid N$.  Let $a\in \mathbb{Z}_N$.  Then  $f(x)$ is induced by  $C_{p^{2m}}(a) $  if and only if   $f^\dagger(x)$ is induced by  $C_{p^{2m}}(-p^ma) $. 
    
\end{lemma}

In view of {\cite[Lemma~2.2]{BJU2019}}, and noting its connection to oddly-good integers, we obtain the following characterizations.

\begin{lemma} \label{charFacScrim} Let $p$ be a prime power and let $N$ be a positive integer  such that $p\nmid N$.   Let $a\in \mathbb{Z}_N$   and $f_a(x)$ be  the monic irreducible polynomial over $\mathbb{F}_{p^{2m}}$ induced by  $C_{p^{2m}}(a)$.  Then the following statements are equivalent.
\begin{enumerate}
    \item  $f_a(x)$ is  a SCRIM factor of $x^N-1$.
    \item $C_{p^{2m}}(a) $  is of type $I'$.
    \item ${\rm ord}(a) \in OG_{({p^{m}},1)}$.
\end{enumerate}
\end{lemma}

For each divisor $d$ of $N$, let $\{a_{d1},  a_{d2},  \dots, a_{d\lambda_d}\}$ be a complete set of representative of the $p^{2m}$-cyclotomic cosets in $A_d$, where    $\lambda_d = \dfrac{\Phi(d)}{{\rm ord}_{d}(p^{2m})}$. 
Using the arguments similar to those for \eqref{facXN-1} and  a  rearrangement  according to the characterization in Lemma \ref{charFacScrim},  we obtain the following factorization which is useful for the subsequent study of Hermitian complementary dual and self-dual cyclic codes.

\begin{align} \label{refacXN-1H}
    x^N-1=  \prod_{\substack{d|N\\ d\in OG_{({p^m},1)}}} \prod _{j=1}^{\lambda_d}  f_{a_{dj}} (x) \prod_{\substack{d|N\\ d\notin OG_{({p^m},1)}}} \prod _{j=1}^{\lambda_d/2}  f_{a_{dj}} (x) f_{a_{dj}} ^\dagger(x),
\end{align}
where  $f_{a_{dj}} ^\dagger(x)= f_{-p^m a_{dj}} (x)$ by Lemma \ref{lem:fdagg}. Consequently, we have 
\begin{align}\label{refacXn-1H}
  x^n-1=  x^{Np^t}-1=  \prod_{\substack{d|N\\ d\in OG_{({p^m},1)}}} \prod _{j=1}^{\lambda_d}  f_{a_{dj}} (x) ^{p^t} \prod_{\substack{d|N\\ d\notin OG_{({p^m},1)}}} \prod _{j=1}^{\lambda_d/2}  f_{a_{dj}} (x)^{p^t} f_{a_{dj}} ^\dagger(x) ^{p^t}.
\end{align}

\begin{example} \label{ex9} The $2^2$-cyclotomic cosets of $\mathbb{Z}_{15}$ are presented in Table \ref{T5} together with their corresponding polynomials over $\mathbb{F}_{2^2}=\{0,1,\alpha,\alpha^2=1+\alpha\}$. 

\begin{table}[!hbt]
    \centering
    \begin{tabular}{ccccc}
    \hline 
        $C_{2^2}(a)$  &  ${\rm ord}(a) $ &${\rm ord}(a) \in OG_{(2,1)}$ & $f_a(x)$ & type \\
        \hline

$\{ 0 \}$ &$1$  & yes &$ x + 1$ &  SCRIM\\

$\{ 1, 4 \}$& $ 15$ & no & $x^2 + x + \alpha$ &  -\\

$\{ 2, 8 \}$ &$ 15 $ &no  &$x^2 + x + \alpha^2 $ & -\\

$\{ 3, 12 \}$ &$ 5 $  & yes &$x^2 + \alpha^2  x + 1$ & SCRIM\\

$\{ 5 \}$ &$ 3 $  & yes &$x + \alpha $ & SCRIM \\

$\{ 6, 9 \}$ &$ 5 $  & yes &$x^2 + \alpha  x + 1$ &  SCRIM\\

$\{ 10 \} $&$3 $  &  yes &$x + \alpha^2 $ & SCRIM\\

$\{ 11, 14 \}$ &$ 15 $   & no &$x^2 + \alpha^2  x + \alpha^2 $ &  -\\

$\{ 13, 7 \}$ &$ 15 $  & no  &$ x^2 + \alpha  x + \alpha  $ & - \\
\hline
    \end{tabular}
    \caption{$2^2$-cyclotomic cosets in $\mathbb{Z}_{15}$ and their corresponding polynomials over $\mathbb{F}_{2^2}$}
    \label{T5}
\end{table}
From Table \ref{T5},  the factorization of $x^{15}-1 $ over $\mathbb{F}_{2^2}$ can be summarized as follows.

\[x^{15}-1= (x + 1)(x^2 + \alpha^2  x + 1)(x + \alpha )(x^2 + \alpha  x + 1)(x + \alpha^2 )\left((x^2 + \alpha  x + \alpha   )(x^2 + x + \alpha)\right)
\left((x^2 + \alpha^2  x + \alpha^2  )(x^2 + x + \alpha^2 )\right),\]
where the first five terms are SCRIM and  the last four terms form two conjugate-reciprocal polynomial pairs. 

\end{example}

\subsubsection{Hermitian Complementary Dual Cyclic Codes over  Finite Fields}

Over a finite field of square order, a \emph{Hermitian complementary dual cyclic code} is defined to be a cyclic code $C$ for which the intersection with its Hermitian dual is trivial.  
In other words, $C$ is Hermitian complementary dual   if  
$
   C \cap C^{\perp_H} = \{0\},
$
where $C^{\perp_H}$ denotes the Hermitian dual of $C$.

This subsection give a precise review on the characterization and construction of  complementary dual cyclic codes under the Hermitian inner product by utilizing the framework of  SCRIM factors of the polynomial $x^N - 1$, together with properties of oddly-good integers.  
Furthermore, an explicit formula for the number of Hermitian complementary dual cyclic codes over finite fields is provided. 
These results can be regarded as a specialized and simplified version of the more general theory of Hermitian complementary dual abelian codes discussed in~\cite{BJU2018}. Alternative,  these can be obtained using the arguments similar  to those in Subsection \ref{secECD} while the definition of Hermitian dual is applied instead of  the Euclidean dual.

The following theorem can be established by employing the definition of Hermitian inner product and arguments analogous to those used in the proof of the Euclidean dual characterization of a cyclic code, as presented in~\cite[Theorem~7.3.7]{LS2004}.

\begin{theorem} \label{genHD}
Let $C$ be a cyclic code of length $n$ over $\mathbb{F}_{p^{2m}}$ with the generator polynomial $g(x)$.  Then the  Hermitian  dual $C^{\perp_{H}}$ of $C$   is $h^\dagger(x)$, where $h(x)= \dfrac{x^n-1}{g(x)}$.
\end{theorem}

Since the intersection $C \cap C^{\perp_H}$ forms a cyclic code generated by   $LCM(g(x), h^\dagger(x))$, the code $C$ is Hermitian complementary dual if and only if $LCM(g(x), h^\dagger(x)) = x^n - 1$. Equivalently, this condition can be expressed in terms of the greatest common divisor, namely $GCD(g(x), h^\dagger(x)) = 1$.  Based on this characterization, the following theorem is established.

\begin{theorem}   Let $p$ be a prime and let $n$ be a positive integer and write $n=Np^t$, where $N$ is a positive integer such that $p\nmid N$ and $t\geq 0$ is an integer. 
    Let $C$ be a cyclic code of length $n$ over $\mathbb{F}_{p^{2m}}$ with generator polynomial $g(x)$. Then $C$ is Hermitian complementary dual if and only  if $g(x)$  is self-conjugate-reciprocal and each irreducible divisor has multiplicity $p^t$.
\end{theorem}

\begin{corollary}  

Let $p$ be a prime  and let $N$ be a positive integer such that $p\nmid N$. 
    Let $C$ be a cyclic code of length $N$ over $\mathbb{F}_{p^{2m}}$ with generator polynomial $g(x)$. Then $C$ is  Hermitian complementary dual if and only  if $g(x)$  is self-conjugate-reciprocal.
\end{corollary}

\begin{theorem} \label{genHLCD}  Let $p$ be a prime and let $n$ be a positive integer and write $n=Np^t$, where $N$ is a positive integer such that $p\nmid N$ and $t\geq 0$ is an integer.   Assume the factorization in \eqref{refacXn-1H} and   let $C$ be a cyclic code of length $n$ over $\mathbb{F}_{p^{2m}}$ with generator polynomial 
\[ g(x) =  \prod_{\substack{d|N\\ d\in OG_{({p^m},1)}}} \prod _{j=1}^{\lambda_d}  f_{a_{dj}} (x) ^{r_{dj} } \prod_{\substack{d|N\\ d\notin OG_{({p^m},1)}}} \prod _{j=1}^{\lambda_d/2}  f_{a_{dj}} (x)^{s_{dj}} f_{a_{dj}} ^\dagger(x) ^{s'_{dj}} .\]
    Then $C$ is Hermitian  complementary dual if and only   if $ r_{dj} \in \{0,p^t\}$ for all  $d|N$ such that $d\in OG_{({p^m},1)}$ and  $j=1,2,\dots, \gamma_d$, and $( s_{dj},s'_{dj}) \in \{(0,0),(p^t,p^t)\}$  for all  $d|N$ such that $d\notin OG_{({p^m},1)}$ and  $j=1,2,\dots, \lambda_d/2$.
\end{theorem}

\begin{theorem} \label{thmNumHLCD}
Let $p$ be a prime and let $n$ be a positive integer and write $n=Np^t$, where $N$ is a positive integer such that $p\nmid N$ and $t\geq 0$ is an integer. Then the number of Hermitian  complementary dual codes of length $n$ over $\mathbb{F}_{p^{2m}}$  is 
    \[ 2^ {\displaystyle \sum_{\substack{d|N\\ d\in OG_{({p^m},1)}}} \frac{\Phi(d)}{{\rm ord}_d (p^{2m})} + \sum_{\substack{d|N\\ d\notin OG_{({p^m},1)}}}  \frac{\Phi(d)}{2{\rm ord}_d (p^{2m})} }.\]
\end{theorem}

\begin{example} From Example \ref{ex9}, 
     we have the factorization of $x^{15}-1 $ over $\mathbb{F}_{2^2}=\{0,1,\alpha,\alpha^2=1+\alpha\}$  of the form 
\[x^{15}-1= (x + 1)(x^2 + \alpha^2  x + 1)(x + \alpha )(x^2 + \alpha  x + 1)(x + \alpha^2 )\left((x^2 + \alpha  x + \alpha   )(x^2 + x + \alpha)\right)
\left((x^2 + \alpha^2  x + \alpha^2  )(x^2 + x + \alpha^2 )\right),\]
where  the first five terms are SCRIM and the last four terms form two conjugate-reciprocal polynomial pairs.
rom Theorem \ref{genHLCD},  it follows that
$(x + 1)^2(x^2 + \alpha^2  x + 1)^2(x + \alpha )^2(x^2 + \alpha  x + 1)^2(x + \alpha^2 )^2$, $(x + 1)^2 
\left((x^2 + \alpha^2  x + \alpha^2  )^2(x^2 + x + \alpha^2 )^2\right)$, and $  \left((x^2 + \alpha  x + \alpha   )^2(x^2 + x + \alpha)^2\right)
\left((x^2 + \alpha^2  x + \alpha^2  )^2(x^2 + x + \alpha^2 )^2\right)$
are generator polynomials  of Hermitian complementary dual cyclic codes of length $30$ over $\mathbb{F}_{2^2}$.  Based on Theorem \ref{thmNumHLCD}, the number of Hermitian  complementary dual codes of length $30$ over $\mathbb{F}_{2^{2}}$  is 
    \[ 2^ {\displaystyle \sum_{\substack{d|15\\ d\in OG_{({2},1)}}} \frac{\Phi(d)}{{\rm ord}_d (2^{2})} +  \sum_{\substack{d|15\\ d\notin OG_{({2},1)}}}   \frac{\Phi(d)}{2{\rm ord}_d (2^{2})} } =  2^ {\displaystyle  \sum_{ d\in \{1,3,5\}}   \frac{\Phi(d)}{{\rm ord}_d (2^{2})} + \sum_{ d\in \{15\}}  \frac{\Phi(d)}{2{\rm ord}_d (2^{2})} } = 2^{(1+2+2)+2} =2^7.\]
\end{example}

 \subsubsection{Hermitian  Self-Dual Cyclic Codes over  Finite Fields}

A cyclic code $C$ of length $n$ over the finite field $\mathbb{F}_{p^{2m}}$ is called \emph{Hermitian self-dual} if it coincides with its Hermitian dual, that is, $C = C^{\perp_H}$.  
Hermitian self-dual cyclic codes have attracted considerable interest in coding theory, not only because of their rich algebraic structure but also due to their applications in error correction and quantum coding.  
A complete characterization, together with explicit enumeration formulas for such codes over finite fields, was provided in \cite[Section~III.B]{JLS2014}.  

In light of \cite[Proposition~2.9]{JLS2014}, the existence of Hermitian self-dual cyclic codes of length $n$ over $\mathbb{F}_{p^{2m}}$ can be described in terms of their generator polynomials.  
More precisely, let $C$ be a cyclic code of length $n$ over $\mathbb{F}_{p^{2m}}$ with generator polynomial $g(x)$.  
Then $C$ is Hermitian self-dual if and only if  
$
   g(x) = h^\dagger(x),
$
where $h(x) = \frac{x^n - 1}{g(x)}$ and $h^\dagger(x)$ denotes the conjugate-reciprocal polynomial of $h(x)$.  
This condition imposes strong restrictions: it can be satisfied only when both the field characteristic and the code length $n$ are even.

\begin{theorem} \label{thm18}
  Let $p$ be a prime and let $n$ be a positive integer and write $n=Np^t$, where $N$ is a positive integer such that $p\nmid N$ and $t\geq 0$ is an integer. Then there exists a Hermitian self-dual cyclic code of length $n$ over $\mathbb{F}_{p^{2m}}$ if and only if $p=2$ and $t\geq 1$. 
 \end{theorem}
 
In this setting, we restrict our attention to the case where the code length is even and can be written in the form  
$
   n = 2^{t} N,
$
where $t$ is a positive integer and $N$ is an odd positive integer.  
The underlying field is assumed to be $\mathbb{F}_{2^{2m}}$ for some positive integer $m$.  
By combining Theorem~\ref{thm18} with the well-known criterion stating that a cyclic code of length $n$ over $\mathbb{F}_{2^{2m}}$ with generator polynomial $g(x)$ is Hermitian self-dual if and only if $g(x) = h^\dagger(x)$, we arrive at the following theorems, which provide a direct characterization of Hermitian self-dual cyclic codes in this case.

\begin{theorem}[{\cite[Theorem 3.9]{JLS2014}}] \label{thm19}
  Let $n$ be an even  positive integer and write $n=N2^t$, where $N$ is an odd positive integer  and $t$ is a positive integer.
Assume the factorization in \eqref{refacXn-1H}  and let $C$ be  a cyclic code  of length $n$ over $\mathbb{F}_{2^{2m}}$ with generator  polynomial $g(x)$.  Then $C$ is Hermitian self-dual if and only if 
  \[ g(x) =  \prod_{\substack{d|N\\ d\in OG_{({2^m},1}}} \prod _{j=1}^{\lambda_d}  f_{a_{dj}} (x) ^{2^{t-1} } \prod_{\substack{d|N\\ d\notin OG_{({2^m},1)}}} \prod _{j=1}^{\lambda_d/2}  f_{a_{dj}} (x)^{s_{dj}} f_{a_{dj}} ^\dagger(x) ^{p^t-s_{dj}} ,\] where $0\leq  {s_{dj}} \leq p^t$ for all   $d|N$ such that $d\notin OG_{({2^m},1)}$ and  $j=1,2,\dots, \gamma_d/2$.
\end{theorem}

Following the result presented in \cite[Corollary~3.7]{JLS2014}, an explicit   formula for the number of   Hermitian self-dual cyclic codes of arbitrary even length $n$ over   $\mathbb{F}_{2^{2m}}$ is derived. This formula is established based on the structural properties of the  the SCRIM factors of  $x^n - 1$   over $\mathbb{F}_{2^{2m}}$ and their link with oddly-good integers in $OG_{(2^m,1)}$.  

\begin{theorem} \label{thm20}
    Let $n$ be an even  positive integer and write $n=N2^t$, where $N$ is an odd positive integer  and $t$ is a positive integer.
Then the number of  Hermitian self-dual  cyclic code  of length $n$ over $\mathbb{F}_{2^{2m}}$ is
\[(1+2^t) ^ {\displaystyle \sum _{ {\substack{d|N\\ d\notin G_{({2^m},1)}}}  }  \frac{\Phi(d)}{ 2{\rm ord}_d(2^{2m})}}.\]
\end{theorem}

\begin{example}
    From Example \ref{ex9}, 
     we have the factorization of $x^{15}-1 $ over $\mathbb{F}_{2^2}=\{0,1,\alpha,\alpha^2=1+\alpha\}$  of the form 
\[x^{15}-1= (x + 1)(x^2 + \alpha^2  x + 1)(x + \alpha )(x^2 + \alpha  x + 1)(x + \alpha^2 )\left((x^2 + \alpha  x + \alpha   )(x^2 + x + \alpha)\right)
\left((x^2 + \alpha^2  x + \alpha^2  )(x^2 + x + \alpha^2 )\right),\]
where  the first five terms are SCRIM and the last four terms form two conjugate-reciprocal polynomial pairs.
From Theorem~\ref{thm19},  it follows that
$ (x + 1)(x^2 + \alpha^2  x + 1)(x + \alpha )(x^2 + \alpha  x + 1)(x + \alpha^2 )\left((x^2 + \alpha  x + \alpha   )(x^2 + x + \alpha)\right)
\left((x^2 + \alpha^2  x + \alpha^2  )(x^2 + x + \alpha^2 )\right)$, $ (x + 1)(x^2 + \alpha^2  x + 1)(x + \alpha )(x^2 + \alpha  x + 1)(x + \alpha^2 )(x^2 + \alpha  x + \alpha   ) ^2
(x^2 + \alpha^2  x + \alpha^2  )^2$, and $ (x + 1)(x^2 + \alpha^2  x + 1)(x + \alpha )(x^2 + \alpha  x + 1)(x + \alpha^2 ) (x^2 + x + \alpha)^2
  (x^2 + x + \alpha^2 )^2  $
are generator polynomials  of Hermitian self-dual cyclic codes of length $30$ over $\mathbb{F}_{2^2}$.  Based on Theorem \ref{thm20}, the number of Hermitian  self-dual codes of length $30=2\cdot 15$ over $\mathbb{F}_{2^{2}}$  is 
   \[(1+2) ^ {\displaystyle \sum _{ {\substack{d|15\\ d\notin G_{({2},1)}}}  }  \frac{\Phi(d)}{ 2{\rm ord}_d(2^{2})}} = (1+2) ^ {\displaystyle \sum _{ d\in \{15\} } \frac{\Phi(d)}{ 2{\rm ord}_d(2^{2})}}=3^2=9.\]
\end{example}

 ~
 
\section{Summary} \label{sec5}

This paper provides a comprehensive review of the concept of good integers, highlighting both their mathematical foundations and their practical significance. It begins with a systematic presentation of number-theoretic properties and characterizations of good integers, including conditions for their existence and divisibility patterns. These results are presented in accessible forms through algorithms and visual aids to support computation and understanding.

The paper then turns to applications in algebraic coding theory, with a particular emphasis on cyclic codes over finite fields. It is demonstrated how the concept of good integers plays a critical role in the characterization and construction of self-dual cyclic codes and complementary dual cyclic codes. These constructions are considered with respect to both the Euclidean and Hermitian inner products, offering insight into how good integers influence duality properties of codes. Several illustrative examples are provided throughout the paper to show how theoretical results translate into concrete coding scenarios, thereby bridging number-theoretic concepts and practical code design.

Beyond the applications discussed in Section~\ref{sec4}, good integers have been further employed in the study of a broad range of algebraic codes. Notably, they have been applied in the construction and analysis of self-dual abelian codes \cite{JLLX2012, JLS2014}, complementary dual abelian codes \cite{BJU2018}, self-dual quasi-abelian codes \cite{JL2015, PJR2018}, and complementary dual quasi-abelian codes \cite{JPD2017}. Moreover, good integers have contributed to the study of code hulls, particularly the hulls of cyclic codes \cite{SJLP2015} and abelian codes \cite{J2018a}. These results highlight the versatility and unifying role of good integers across various code families and algebraic settings.  The reader is referred to the works cited above for such applications.

This survey emphasizes the unifying role of good integers in bridging number theory with   applications in coding.   The study of additional applications of good integers, either within coding theory or in other mathematical and computational disciplines, presents a promising direction for future research. Exploring their roles in new algebraic frameworks, combinatorial structures, or cryptographic systems may yield novel insights and broaden the impact of their theory.

%%%%%%%%%%%%%%%%%%%%%%%%%%

\section*{Acknowledgments}
This research was funded by the National Research Council of Thailand and Silpakorn University  under Research Grant  N42A650381.

\bibliographystyle{abbrv}
\bibliography{BIB-AMS}
\end{document}